\documentclass[11pt]{article}

\usepackage{times,amsmath,amsthm,amssymb,graphicx,xspace,epsfig,xcolor}
\usepackage{lineno}
\usepackage{graphicx,multicol}
\usepackage{subfigure}

\theoremstyle{plain}
 \newtheorem{theorem}{Theorem}
 \newtheorem{lemma}[theorem]{Lemma}
 \newtheorem{corollary}[theorem]{Corollary}
\newtheorem{proposition}[theorem]{Proposition}
\newtheorem{claim}{Claim}[theorem]

\theoremstyle{definition}

 \newtheorem{conjecture}[theorem]{Conjecture}

\newcommand{\inv}{\overleftarrow}

\newcommand{\pb}{{\sc Subdivision}}

\newenvironment{subproof}{\par\noindent {\it Subproof}.\ }{\hfill$\lozenge$\par\vspace{11pt}}

\title{Finding a subdivision of a prescribed digraph of order $4$}

\author{
Fr\'ed\'eric Havet\thanks{Partially supported by the ANR project STINT.}~~\thanks{This work was mainly done when F. Havet was on leave at Simon Fraser University, PIMS, UMI 3069, CNRS.}
\\
COATI Project, I3S (CNRS/UNSA) {\scriptsize \&} INRIA
\\ 2004 route des Lucioles BP 93,\\
06902 Sophia-Antipolis Cedex, France\\
 email: {\tt frederic.havet@cnrs.fr}
\and A. Karolinna Maia\thanks{Partially supported by CAPES/Brazil.}\\
ParGO, Department of Computer Science, UFC\\
Campus do Pici, Bloco 910,\\
60455760 - Fortaleza, CE, Brazil\\
 email: {\tt karolmaia@lia.ufc.br}
\and
Bojan Mohar\thanks{Supported in part by the
  Research Grant P1--0297 of ARRS (Slovenia), by an NSERC Discovery Grant (Canada)
  and by the Canada Research Chair program.}~\thanks{On leave from:
  IMFM \& FMF, Department of Mathematics, University of Ljubljana, Ljubljana,
  Slovenia.}\\
Department of Mathematics, Simon Fraser University\\
  Burnaby, B.C. V5A 1S6 \\
  email: {\tt mohar@sfu.ca}
}

\begin{document}

\maketitle

\begin{abstract}
The problem of when a given digraph contains a subdivision of a fixed digraph $F$ is considered.
Bang-Jensen et al.\ \cite{BHM} laid out foundations for approaching this problem from the algorithmic point of view. In this paper we give further support to several open conjectures and speculations about algorithmic complexity of finding $F$-subdivisions.
In particular, up to 5 exceptions, we completely classify for which 4-vertex digraphs $F$, the $F$-subdivision problem is polynomial-time solvable and for which it is NP-complete.
While all NP-hardness proofs are made by reduction from some version of the 2-linkage problem in digraphs,
some of the polynomial-time solvable cases involve relatively complicated algorithms.
\end{abstract}


\section{Introduction}

In this paper, all digraphs are meant to be {\em strict}, that is without loops and without multiple arcs. In one occasion, however, multiple arcs will be allowed. In that case, we will use the term {\em multidigraph}. We follow standard terminology as used in \cite{livre:digraph,BoMu08}.

A {\it subdivision of a digraph $F$}, also called an {\it $F$-subdivision}, is a digraph obtained from $F$ by replacing each
arc $ab$ of $F$ by a directed $(a,b)$-path.
In this paper, we consider the following problem for a fixed digraph $F$.

\medskip

\noindent {\sc $F$-Subdivision}\\
\underline{Input}: A digraph $D$.\\
\underline{Question}: Does $D$ contain a subdivision of $F$ as a subdigraph?

\medskip

Bang-Jensen et al.~\cite{BHM} conjectured that there is a dichotomy between NP-complete and polynomial-time solvable instances.
\begin{conjecture}\label{dichotomy}
For every digraph $F$, the $F$-\pb\ problem is polynomial-time solvable or NP-complete.
\end{conjecture}

According to this conjecture, there are only two kinds of digraphs $F$:
{\it intractable} digraphs $F$, for which $F$-\pb\ is NP-complete, and {\it tractable} digraphs, for which $F$-\pb\ is solvable in polynomial-time.

Bang-Jensen et al.~\cite{BHM} proved that many digraphs are intractable; see Theorem~\ref{NP-big} in Section~\ref{sec:hard}.
In particular, every digraph in which every vertex $v$ is {\it big} (that is such
that either $d^+(v)\geq 3$, or $d^-(v)\geq 3$, or $d^-(v)=d^+(v)=2$) is intractable.
They also give many examples of tractable digraphs. See Subsection~\ref{sec:known}.
However, there is no clear evidence, of which graph should be tractable and which one should be intractable, despite some results and conjectures give some outline.

Establishing a conjecture of of Johnson et al. \cite{JRST82}, Kawabarayashi and Kreutzer~\cite{KaKr14} proved the Directed Grid Theorem.

\begin{theorem}[Kawabarayashi and Kreutzer~\cite{KaKr14}]\label{thm:DGT}
For any positive integer $k$, there exists an integer $f(k)$ such that every digraph with directed treewidth greater than $f(k)$ contains a cylindrical grid of order $k$ as a butterfly minor.
\end{theorem}

Here, a {\it cylindrical grid} of order $k$ consists of $k$ concentric directed cycles and $2k$ directed paths connecting the cycles in alternating directions. See Figure~\ref{fig:grid} for an illustration. A {\it butterfly minor} of a digraph $D$ is a digraph obtained from a subgraph of $D$ by contracting arcs which are either the only outgoing arc of their tail or the only incoming arc of their head.

\begin{figure}[htb]
       \centering
       \includegraphics[height=3.8cm]{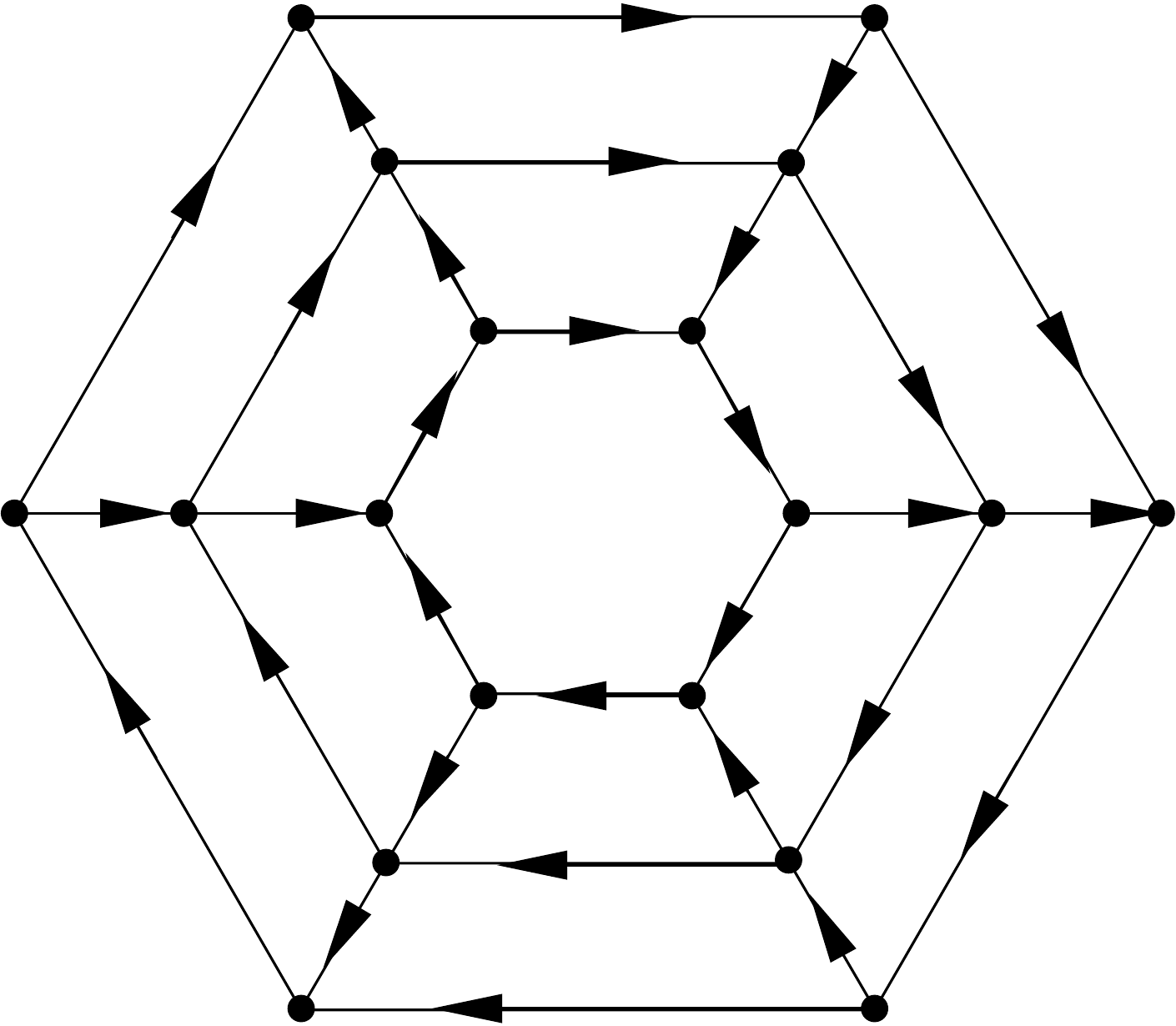}
       \caption{The cylindrical grid of order $3$.}
       \label{fig:grid}
\end{figure}

Moreover, their proof is algorithmic.
\begin{theorem}[Kawabarayashi and Kreutzer~\cite{KaKr14}]\label{thm:DGT-algo}
For any positive integer $k$, there exists an integer $f(k)$ such that given any digraph, in polynomial time, we can obtain either
\begin{itemize}
\item a cylindrical grid of order $k$ as a butterfly minor, or
\item a directed tree decomposition of width at most $f(k)$.
\end{itemize}
\end{theorem}

Because the $k$-{\sc Linkage} problem (See Section~\ref{sec:hard} for definitions and details on $k$-{\sc Linkage}) is polynomial-time solvable on digraph with bounded directed treewidth, for any fixed $F$, the $F$-\pb\ is also polynomial-time solvable on digraphs of bounded directed tree-width (see \cite{BHM} for more details.).
Moreover, by induction on the number of vertices, one can show that for any  planar digraph with no big vertices $F$, there is an integer $k_F$ such that the cylindrical grid of order $k_F$ contains an $F$-subdivision. Furthermore, if $F$ has no big vertices, then if a minor of $D$ has an $F$-subdivision, then so does $D$. (One can uncontract the arcs without any problem because the vertices of $F$ are not big).
All these directly imply the following:

\begin{corollary}\label{cor:nobig}
$F$-\pb\ is polynomial-time solvable when $F$ is a planar digraph with no big vertices.
\end{corollary}
\begin{proof} On can solve $F$-\pb\ as follows. Given a digraph $D$, one runs the algorithm given by Theorem~\ref{thm:DGT-algo}. If it returns a cylindrical grid of order $k_F$, then we return `Yes' as it contains an $F$-subdivision.
If it returns a directed tree decomposition of width at most $f(k_F)$, then one runs the polynomial-time algorithm to solve $F$-subdivision for digraphs with directed tree width at most $f(k_F)$.
\end{proof}

On the other hand, Bang-Jensen et al. \cite{BHM} proposed the following sort of counterpart.

\begin{conjecture}[Bang-Jensen et al. \cite{BHM}] \label{conj:nonplanar}
$F$-\pb\ is NP-complete for every non-planar digraph~$F$.
\end{conjecture}

Bang-Jensen et al.~\cite{BHM} were able to classify all digraphs of order at most $3$: they are all tractable except the complete symmetric digraph on three vertices, which is intractable. In this paper, we consider digraphs of order $4$.
We classify all digraphs of order $4$ except for five of them (up to directional duality).
These are the digraphs $O_i$ for $1\leq i\leq 5$ depicted Figure~\ref{fig:open}.

\begin{figure}[htb]
       \centering
       \includegraphics[height=1.8cm]{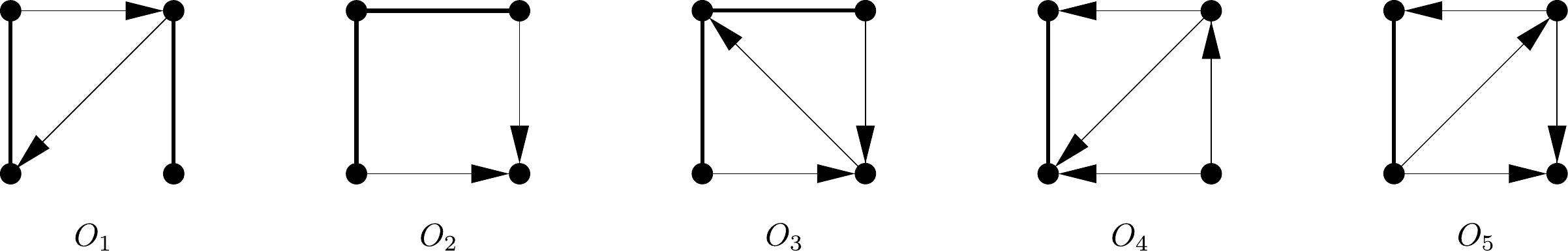}
       \caption{Digraphs on $4$-vertices that are not known to be tractable or intractable. Bold undirected edges represent directed $2$-cycles.}
       \label{fig:open}
\end{figure}

\begin{theorem}\label{thm:main}
Let $F$ be a digraph of order $4$ that is not isomorphic to any of $O_j$ for $1\leq j\leq 5$.\\
 If  $F$ contains a directed 2-cycle whose vertices are big  or $F$ is one of the graphs $N_i$  depicted in Figure~\ref{fig:NP} for $1\leq i\leq 9$, then {\sc $F$-Subdivision} NP-complete.
 Otherwise {\sc $F$-Subdivision} is polynomial-time solvable.
\end{theorem}

\begin{figure}[htb]
       \centering
       \includegraphics[height=5.5cm]{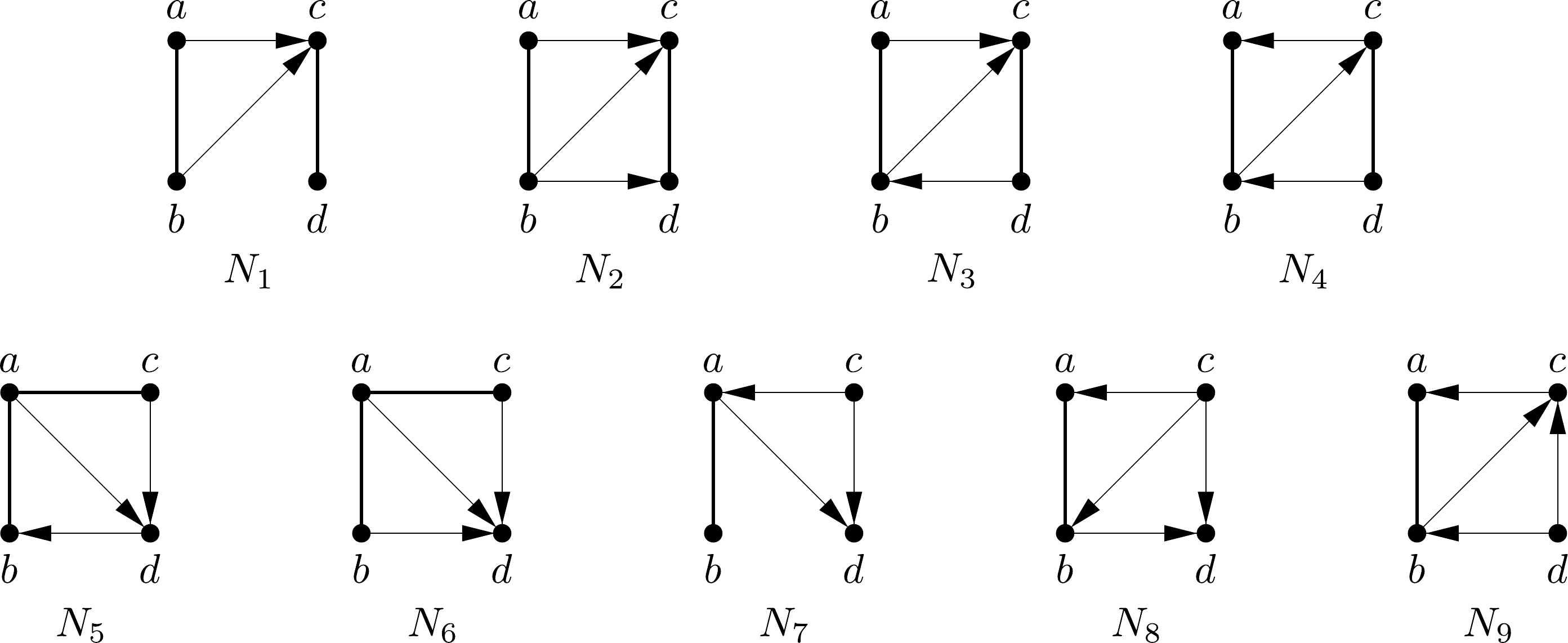}
       \caption{Some intractable digraphs on $4$-vertices. Bold undirected edges represent directed $2$-cycles.}
       \label{fig:NP}
\end{figure}

Theorem~\ref{thm:main} implies that all oriented graphs of order $4$ are tractable.
In particular, the wheel $W_3$ is tractable. The {\it wheel} $W_k$ is the graph obtained from the directed cycle on $k$ vertices $\vec{C}_k$ by adding a vertex, called the {\it centre}, dominating every
vertex of $\vec{C}_k$.
In \cite{BHM}, Bang-Jensen et al. proved that $W_2$ is tractable and that
$W_k$-\pb\ is NP-complete for all $k\geq 4$.
The case of $W_3$ was left as an open problem.

\begin{figure}[htb]
       \centering
       \includegraphics[height=1.6cm]{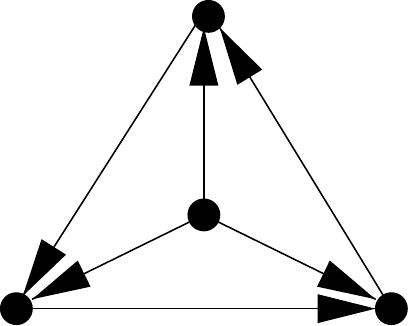}
       \caption{The $3$-wheel $W_3$.}
       \label{fig:W3C4}
\end{figure}

Theorem~\ref{thm:main} also completes the classification of tournaments. Bang-Jensen et al.~\cite{BHM} proved that every tournament of order at most $3$ is tractable, and that every tournament of order at least $5$ is intractable (see Theorem~\ref{NP-big}). They also show that the transitive tournament of order $4$ is tractable. The other tournaments of order four are $W_3$, its converse, and $ST_4$, the strong tournament of order $4$, no vertex of which is big.

\medskip

In Section~\ref{sec:hard} we prove some digraphs $F$ to be intractable. To do so, we use a reduction from the NP-complete problem {\sc Restricted $2$-Linkage}; given a digraph $D$ without big vertices in which $x_1$ and $x_2$ are sources and $y_1$ and $y_2$ are sinks, this problem consists in deciding whether there exists two vertex-disjoint directed paths $P_1$ from $x_1$ to $y_1$ and $P_2$ from $x_2$ to $y_2$.
The proofs all use the same technique. Given an instance $D, x_1, x_2, y_1,y_2$ of {\sc Restricted $2$-Linkage}, we construct a digraph $D'$ by putting $D$ on two arcs $e_1 = u_1v_1$ and $e_2 = u_2v_2$ of $F$, that is by taking the disjoint union of $D$ and $F$, removing the arcs $e_1$ and $e_2$ and adding the arcs $u_1x_1$, $y_1v_1$, $u_2x_2$ and $y_2v_2$ and show that $D'$ an $F$-subdivision if and only if there is a $2$-linkage from $(x_1, x_2)$ to $(y_1, y_2)$ in $D$. This implies that $F$-\pb\
 is NP-complete
 Unfortunately, this technique does not work for $O_j$, $1\leq j\leq 5$. For any pair of arcs $e_1$ and $e_2$ of $O_j$, the existence of an $F$-subdivision in the digraph $D'$ obtained by putting $D$ on those arcs does not imply the existence of a $2$-linkage from $(x_1, x_2)$ to $(y_1, y_2)$ in $D$.

\medskip
We then turn out to prove some digraphs of order $4$ to be tractable.
Since every digraph of order $4$ is planar, Corollary~\ref{cor:nobig} implies that digraphs of order $4$ with no big vertices are tractable.
Before the Directed Grid Theorem was proved, we found elementary proofs to show that  result. These proofs can be easily implemented and the derived algorithms are certainly of lower complexity than the ones derived from the Directed Grid Theorem. They can be found in \cite{HMM}.
Consequently, we only consider digraphs with at least one big vertex.

We present in Section~\ref{sec:known} some known results and tools, including Menger's Theorem, which we profusely use.
Then we scan the digraphs $F$ of vertices with respect to the number of edges in their $2$-cycle digraph $G_F$, which is the graph with vertex set $V(F)$, in which two vertices are linked by an edge if they are in a directed 2-cycle in $F$. Corollary~\ref{NP-2cycle} proved in Section~\ref{sec:hard}, implies that if $G_F$ has three edges, then $F$ is intractable. So we only need to examine digraphs $F$ for which $G_F$ has at most two edges.
We first consider the oriented graphs, (for which $G_F$ has no edge). We prove in in Section~\ref{sec:oriented} that all oriented graphs of order $4$ are tractable. The main result here is a polynomial-time algorithm solving $W_3$-\pb\ (Theorem~\ref{3wheel}).
Next we consider digraphs $F$ for which $G_F$ has one or two edges). We prove that the digraphs $E_i$,  $1\leq i\leq 9$,  depicted in Figure~\ref{fig:easy} are tractable.
\begin{figure}[htb]
       \centering
       \includegraphics[height=5.5cm]{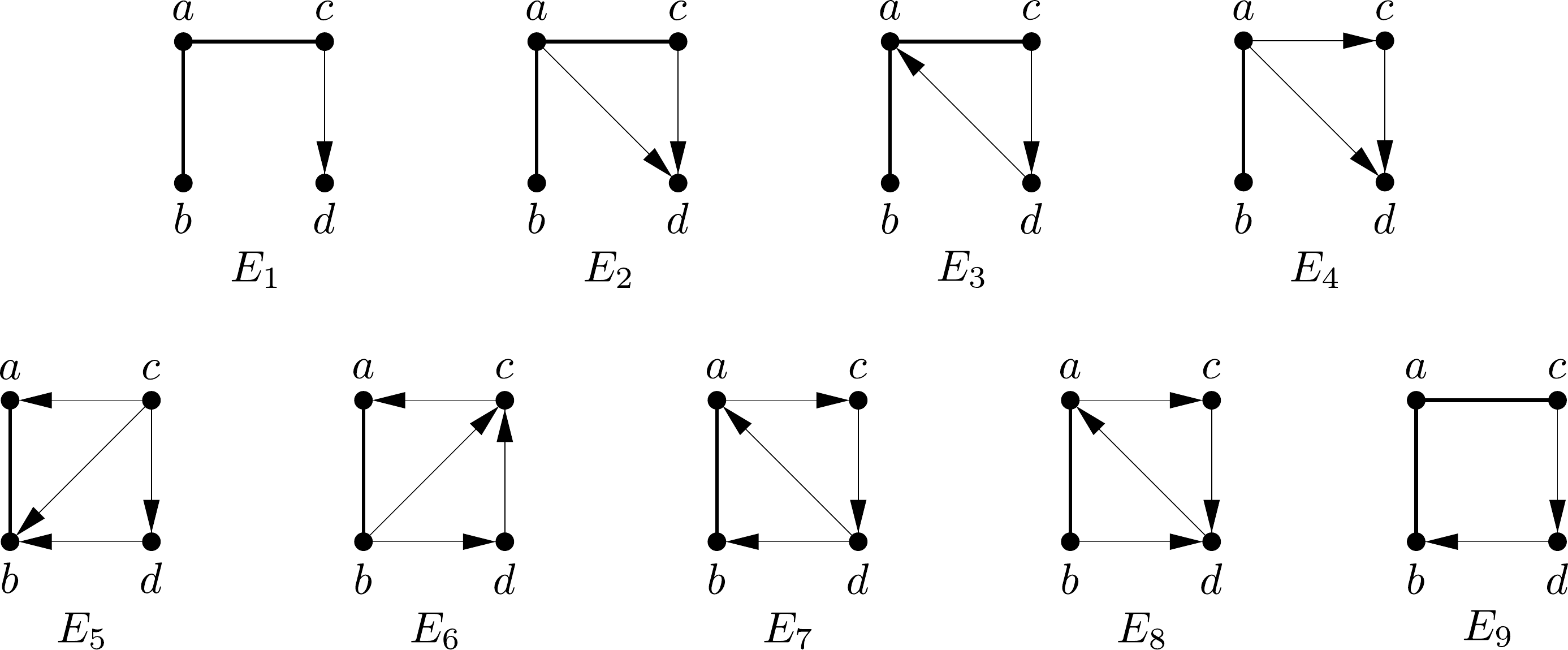}
       \caption{Some tractable digraphs on $4$-vertices. Bold undirected edges represent directed $2$-cycles.}
       \label{fig:easy}
\end{figure}
The polynomial-time algorithm to solve $E_i$-\pb\ is relatively easy for $1\leq i\leq 8$, but the one to solve $E_9$ is more involved.
Finally, in Section~\ref{sec:classify}, we summarize all results to prove Theorem~\ref{thm:main}.

\subsection{Finding an $F$-subdivision}

The letters $n$ and $m$ will always denote the number of vertices and arcs of the input digraph $D$ of the problem in question. By {\it linear time}, we mean $O(n+m)$ time.

\begin{lemma}\label{lem:find}
If $F$-\pb\ can be solved in $f(n,m)$  time,  where $f$ is non-decreasing in $m$, then
there is an algorithm that finds an $F$-subdivision (if one exists) in a digraph in $((m+1)\cdot f(n,m)+m)$ time.
\end{lemma}
\begin{proof}
Suppose that there exists an algorithm {\tt $F$-decide}$(D)$ that decides in $f(n,m)$ whether $D$ contains an $F$-subdivision.
We now construct an algorithm {\tt $F$-find}$(D)$ that finds an $F$-subdivision in $D$ if there is one, and returns `no' otherwise.
It proceeds as follows.

Let $a_1, \dots , a_m$ be the arcs of $D$.
If {\tt $F$-decide}$(D)$ returns `no', then we also return `no'.
If not, then $D$ contains an $F$-subdivision, we find it as follows:
We initialize $D_0:=D$.
For $i=1$ to $m$, $D_{i}:= D_{i-1}\setminus a_i$
if   {\tt $F$-decide}$(D_{i-1}\setminus a_i)$ returns `yes', and $D_{i}:= D_{i-1}$ otherwise.

{\tt $F$-find} is valid because at step $i$, we delete the arc $a_i$ if and only if there is an $F$-subdivision not containing $i$.
Hence at each step $i$, we are sure that $D_i$ contains an $F$-subdivision, and that any $F$-subdivision must contain all the arcs
of $A(D_i)\cap \{a_1, \dots ,a_i\}$.

{\tt $F$-find} runs $(m+1)$ times the algorithm {\tt $F$-decide} and removes at most $m$ times an arc.
Therefore, it runs in time $(m+1)\cdot f(n,m)+m$.
\end{proof}

Lemma~\ref{lem:find} implies that deciding if there is an $F$-subdivision in a digraph is polynomial-time solvable, if and only if, finding an $F$-subdivision in a digraph is polynomial-time solvable. Therefore,
since we are primarily interested in determining if the problems are polynomial-time solvable or NP-complete, and
for sake of clarity, we only present algorithms for solving $F$-\pb\ as a decision problem. However, the proofs of validity of all given algorithms always rely on constructive claims. Hence each algorithm can be easily transformed into a polynomial-time algorithm for finding an $F$-subdivision in a given digraph.
Moreover, the reader can check that the additional work does not increase the time complexity. Hence, our algorithms for finding $F$-subdivisions have the same complexity as their decision versions.

\section{Definitions and notations}

We rely on \cite{livre:digraph,BoMu08} for standard notation and concepts.
Let $D$ be a digraph. The {\it converse} of $D$ is the digraph $\overline{D}$ obtained from $D$ by reversing the orientation of all arcs.
We denote by $UG(D)$ the underlying (multi)graph of $D$, that is, the
(multi)graph we obtain by replacing each arc by an edge.
To every graph $G$, we can {\it associate} a symmetric digraph by replacing every edge $uv$ by the two arcs $uv$ and $vu$.

A {\it source} in $D$ is a vertex of indegree zero and a {\it sink} is a vertex of outdegree zero.

An {\it oriented graph} is an orientation of an undirected graph. In other words, it is a digraph with no directed cycles of length $2$.
An {\it oriented path} is an orientation of an undirected path.
Hence an oriented path $P$ is a sequence $(x_1, a_1, x_2, a_2, \dots , a_{n-1}, x_n)$, where the $x_i$ are distinct vertices and for all $1\leq j\leq n-1$, $a_j$ is either the arc $x_jx_{j+1}$ or the arc $x_{j+1}x_{j}$.
For sake of clarity, we often refer to such an oriented path $P$ by the underlying undirected path $x_1x_2 \dots x_n$.
This is a slight abuse, because the oriented path $P$ is not completely determined by this sequence as there are two possible orientations for each edge. However, when we use this notation, either the orientation does not matter or it is clear from the context.

Let $P=x_1x_2\cdots x_n$ be an oriented path.
We say that $P$ is an {\it $(x_1,x_n)$-path}.
The vertex $x_1$ is the {\it initial vertex} of $P$ and $x_n$ its {\it terminal vertex}.
We denote the initial vertex of $P$ by $s(P)$ and the terminal vertex of $P$ by $t(P)$. The subpath $x_2\cdots x_{n-1}$ is denoted by $P^{\circ}$.
If $x_1x_2$ is an arc, then $P$ is an {\it outpath}, otherwise $P$ is an {\it inpath}. The path $P$ is {\it directed} if no vertex is the tail of two arcs in $P$ nor the head of two arcs. In other words, all arcs are oriented in the same direction.
There are two kinds of directed paths, namely directed outpaths and directed inpaths.
For convenience, a directed outpath is called a {\it dipath}.
The {\it blocks} of an oriented path $P$ are the maximal directed subpaths of $P$. We often enumerate
them from the initial vertex to the terminal vertex of the path. The number of blocks of $P$ is denoted by $b(P)$. The {\it opposite path} of $P$, denoted $\inv{P}$, is the path $x_nx_{n-1}\cdots x_1$. For $1\leq i\leq j\leq n$, we denote by $P[x_i, x_j]$ (resp. $P]x_i, x_j[$, $P]x_i, x_j]$, $P[x_i, x_j[$), the oriented subpath $x_ix_{i+1} \dots x_j$ (resp. $x_{i+1}x_{i+2} \dots x_{j-1}$, $x_{i+1}x_{i+2} \dots x_j$, $x_ix_{i+1} \dots x_{j-1}$).

The above definitions and notation can also be used for oriented cycles. If $C=x_1x_2 \dots x_nx_1$ is an oriented cycle, we shall assume that either $C$ is a {\it directed cycle}, that is $x_ix_{i+1}$ is an arc for all $1\leq i\leq n$, where $x_{n+1}=x_1$, or both edges of $C$ incident with $x_1$ are directed outwards, i.e. $x_1x_2$ and $x_1x_n$ are arcs of $C$.

For a set $X$ of vertices, the {\it outsection} of $X$ in $D$, denoted by $S_D^+(X)$, is the set of vertices that are reachable from $X$ by a dipath. The outsection of a set in a digraph can be found in linear time using the Breadth-First Search.
The directional dual notion, the {\it insection} of $X$ in $D$ is denoted by $S_D^-(X)$.

The digraph $D$ is {\it connected} (resp.\ {\it $k$-connected}) if $UG(D)$ is a connected (resp.\ $k$-connected) graph.
It is {\it strongly connected}, or {\it strong}, if for any two
vertices $u,v$, there is a $(u,v)$-dipath in $D$.
If $D$ is strong, we use the notation $D[u,v]$ to denote any $(u,v)$-dipath in $D$.
The disjoint union of two digraphs $D_1$ and $D_2$ is denoted $D_1+D_2$.

By {\it contracting} a vertex-set $X\subseteq V(D)$ we refer to the operation of first taking the digraph $D-X$ and then adding new vertex $v_X$ and adding the arc
$v_Xw$ for each $w \in V(D-X)$ with an inneighbour in $X$ and the arc $uv_X$ for each $u \in V(D-X)$ with an outneighbour in $X$.
The {\it contraction} of a non-strong digraph $D$ is the digraph obtained by contracting all strong components of $D$.

\section{Intractable digraphs}\label{sec:hard}

Let $x_1, x_2, \dots , x_k, y_1, y_2, \dots , y_k$ be distinct vertices of a digraph $D$. A
{\it $k$-linkage} from $(x_1, x_2,\dots , x_k)$ to $(y_1, y_2, \dots , y_k)$ in $D$ is a system of disjoint
dipaths $P_1, P_2, \dots , P_k$ such that $P_i$ is an $(x_i, y_i)$-path in $D$.
Fortune, Hopcroft and Wyllie~\cite{FHW80} showed that for any $k\geq 2$, {\sc $k$-Linkage} is NP-complete.
The problem is also NP-complete when restricted to some classes of digraphs.
Recall that a vertex $v$ is {\it big} if either $d^+(v)\geq 3$, or $d^-(v)\geq 3$, or $d^-(v) = d^+(v)= 2$.

\bigskip

\noindent {\sc Restricted $2$-Linkage}\\
\underline{Input}: A digraph $D$ without big vertices in which $x_1$ and $x_2$ are sources and $y_1$ and $y_2$ are sinks. \\
\underline{Question}: Is there a $2$-linkage from $(x_1, x_2)$ to
$(y_1, y_2)$ in $D$?

\begin{theorem}[Bang-Jensen et al.~\cite{BHM}]\label{NP-small}
The {\sc Restricted $2$-Linkage} problem is NP-complete.
\end{theorem}

Using this theorem, Bang-Jensen et al.~\cite{BHM} deduced a sufficient condition for $F$-\pb\ to be NP-complete.

For a digraph $D$, we denote by $B(D)$ the set of its big vertices.
A {\it big path} in a digraph is a directed path whose endvertices are big and whose internal vertices all have both indegree and outdegree equal to $1$ in $D$ (in particular an arc between two big vertices is a big path). Note also that two big paths with the same endvertices are necessarily internally disjoint.

The {\it big paths digraph} of $D$, denoted $BP(D)$, is the multidigraph
with vertex set $B(D)$ in which there are as many arcs between two vertices $u$ and $v$
as there are big $(u,v)$-paths in $D$.

\begin{theorem}[Bang-Jensen et al.~\cite{BHM}]\label{NP-big}
Let $F$ be a digraph.
If $F$ contains two arcs $ab$ and $cd$ whose endvertices are big vertices and such that
$(BP(F)\setminus \{ab, cd\}) \cup \{ad, cb\}$ is not isomorphic to $BP(F)$, then $F$-\pb\ is NP-complete.
\end{theorem}

\begin{corollary}\label{NP-2cycle}
Let $F$ be a digraph.
If $F$ contains a directed cycle of length $2$ whose vertices are big, then
$F$-\pb\ is NP-complete.
\end{corollary}

So far, all known intractable digraphs were proved intractable by a reduction from {\sc Restricted $2$-Linkage}.
This paper is no exception:  we now show that some digraphs are NP-hard with such reductions.

\begin{proposition}\label{propNP}
For each digraph $N_i$, $1\leq i\leq 9$, depicted in Figure~\ref{fig:NP}, $N_i$-\pb\ is NP-complete.
\end{proposition}

\begin{proof}
In each case, the problem is proved to be NP-complete by reduction from {\sc Restricted $2$-Linkage}.
Let $D$, $x_1$, $x_2$, $y_1$ and $y_2$ be an instance of this problem.
We construct a digraph $D_i$ by {\it putting} $D$ on two arcs $e_1=u_1v_1$ and $e_2=u_2v_2$ of $N_i$ (that will be specified later), that is by taking the disjoint union of $D$ and $N_i$, removing the
arcs $e_1$ and $e_2$ and adding the arcs $u_1x_1$, $y_1v_1$, $u_2x_2$ and $y_2v_2$.
We then show that $D_i$ contains an $N_i$-subdivision if and only if there is a $2$-linkage from $(x_1, x_2)$ to $(y_1, y_2)$ in $D$.
This implies that $N_i$-\pb\ is NP-complete.

Clearly, by construction of $D_i$, if there is a $2$-linkage from $(x_1, x_2)$ to $(y_1, y_2)$ in $D$, then $D_i$ contains an $N_i$-subdivision.
We now prove the converse for each $i$. In each case we shall assume that $D_i$ contains an $N_i$-subdivision $S$, and we shall denote by $a',b',c',d'$ the vertices in $S$ corresponding to
$a,b,c,d$, respectively.

\smallskip\noindent $i=1$:~
We choose $e_1=ab$ and $e_2=cd$.
Since $D$ contains no big vertices, we have $c'=c$.
Because $d^-_{D_1}(c)=3$, the arcs $ac$, $bc$ and $dc$ are in $S$. Moreover, the arc $ba$ is in $S$, because every vertex has indegree at least $1$ in $S$.
Thus $d^+_S(b)\geq 2$, and so either $b=b'$ or $b=a'$. By symmetry between $a$ and $b$ in $N_1$, we may assume that $b=b'$. Then, necessarily, $a=a'$.
Therefore, in $S$, there are disjoint $(a,b)$- and $(c,d)$-dipaths.
These two paths induce a $2$-linkage from $(x_1, x_2)$ to $(y_1, y_2)$ in $D$.

\smallskip\noindent $i\in\{2, 3, 4\}$:~
We choose $e_1=ab$ and $e_2=cd$.
Since $D$ contains no big vertices, we have $\{b,c\} = \{b',c'\}$. Therefore,
the arc $bc$ is contained in $S$, and this shows that $b'=b$ and $c'=c$.
Now for degree reasons, all arcs incident to $b$ and $c$ must be in $S$.
It follows that $a'=a$ and $d'=d$. (This is clear for $N_3$ and $N_4$. For $N_2$, we first conclude that $\{a',d'\}=\{a,d\}$ and then consider degrees of $a$ and $d$ to obtain the same conclusion.)
Therefore, in $S$, there are disjoint $(a,b)$- and $(c,d)$-dipaths.
These two paths induce a $2$-linkage from $(x_1, x_2)$ to $(y_1, y_2)$ in $D$.

\smallskip\noindent $i=5$:~
We choose $e_1=ba$ and $e_2=cd$.
Since $D$ contains no big vertices, we have $a'=a$.
Hence all the arcs incident to $a$ ($ac, ca, ad, ab, y_1a$) are in $A(S)$. Therefore, since $aca$ is a $2$-cycle and it is in $S$, $c$ is either $b'$ or $c'$.
But $d^-(c)=1$, so $c$ cannot be $b'$ since $b'$ must have indegree at least $2$, and thus $c=c'$.
All vertices have outdegree at least $1$ in $S$, so $db \in A(S)$ since we know that the arc $ad$ and therefore $d$ is in $S$.
Now there are two internally disjoint $(a',b)$-dipaths in $S-c'$, then necessarily $b=b'$, since there should be two disjoint paths from $a'$ to $b'$ in $S-c'$ and they should use the arcs $ad$ and $ab$ that we have already conclude they must belong to $S$. Moreover, $d'$ must be in one of those dipaths, so $d=d'$.
Therefore, in $S$, there are internally disjoint $(b,a)$- and $(c,d)$-dipaths.
These two paths induce a $2$-linkage from $(x_1, x_2)$ to $(y_1, y_2)$ in $D$.

\smallskip\noindent $i=6$:~
We choose $e_1=ab$ and $e_2=cd$.
Since $D$ contains no big vertices, we have $a'=a$ and $d'=d$.
Hence all arcs incident to those two vertices are in $S$. Therefore $\{b',c'\}=\{b,c\}$. By symmetry of $N_6$, we may assume that $b'=b$ and $c'=c$.
Therefore, in $S$, there are disjoint $(a,b)$- and $(c,d)$-dipaths.
These two paths induce a $2$-linkage from $(x_1, x_2)$ to $(y_1, y_2)$ in $D$.

\smallskip\noindent $i=7$:~
We choose $e_1=ab$ and $e_2=cd$.
Since $D$ contains no big vertices, we have $a'=a$.
Hence all arcs incident to $a$ are in $S$.
So $c$ and $d$ are in $V(S)$.
Since $d^+_{D_7}(d)=0$, we have $d=d'$;
since $d^-_{D_7}(c)=0$, we have $c=c'$.
Therefore, in $S$, there are disjoint $(a,b)$- and $(c,d)$-dipaths.
These two paths induce a $2$-linkage from $(x_1, x_2)$ to $(y_1, y_2)$ in $D$.

\smallskip\noindent $i=8$:~
We choose $e_1=ab$ and $e_2=cd$.
Since $D$ contains no big vertices, we have $b'=b$ and $c'=c$.
Hence all arcs incident to those two vertices are in $S$.
So $d\in V(S)$. Since $d^+_{D_8}(d)=0$, it follows that $d=d'$.
The arcs $ba$ and $ca$ show that $d^-_S(a)\ge2$. Thus $a=a'$.
Therefore, in $S$, there are disjoint $(a,b)$- and $(c,d)$-dipaths.
These two paths induce a $2$-linkage from $(x_1, x_2)$ to $(y_1, y_2)$ in $D$.

\smallskip\noindent $i=9$:~
We choose $e_1=ab$ and $e_2=dc$.
Since $D$ contains no big vertices, we have $b'=b$.
Hence all arcs incident to $b$ are in $S$.
In particular $c,d\in V(S)$. Since $d^-_{D_{9}}(d)=0$, we have $d'=d$.
Since $d^+_S(c)\geq 1$, the arc $ca$ is in $A(S)$, so $d^-_S(a)=2$, and thus $a\in \{a',c'\}$.
Since $a'$ and $c'$ are both in the outsection of $d$ in $N_{9}-b$, $S$ contains a $(d,a)$-dipath
disjoint from $b$. This dipath must pass through $c$ and therefore the arc $y_2c$ lies in $S$.
This implies that $d^-_S(c)\ge2$, so $c=c'$ and then we have $a=a'$.
Consequently, in $S$, there are disjoint $(a,b)$- and $(d,c)$-dipaths.
These two paths induce a $2$-linkage from $(x_1, x_2)$ to $(y_1, y_2)$ in $D$.
\end{proof}

\section{Known results and tools for $F$-\pb}\label{sec:known}

\subsection{Menger's Theorem}

Let $X$ and $Y$ be two sets of vertices in a digraph $D$.
An {\it $(X,Y)$-dipath} is a dipath with initial vertex in $X$, terminal vertex in $Y$ and all
internal vertices in $V(D)\setminus (X\cup Y)$.
For notational clarity, for a vertex $x$ (resp. a subdigraph $S$ of $D$), we abbreviate $\{x\}$ to $x$ (resp. $V(S)$ to $S$) in the notation. For example, an $(x,S)$-dipath is an $(\{x\}, V(S))$-dipath.

Let $D$ be a digraph, and let $x$ and $y$ be distinct vertices of $D$.
Two $(x,y)$-paths $P$ and $Q$ are {\it internally disjoint} if they have no internal vertices in common, that is, if $V(P) \cap V(Q)= \{x,y\}$.
A {\it $k$-separation} of $(x,y)$ in $D$ is a partition $(W,S,Z)$ of its vertex set such that $x\in W$, $y\in Z$, $|S| \leq k$, each vertex in $W$ can be reached from $x$ by a dipath in $D[W]$, and there is no arc from $W$ to $Z$.

One version of the celebrated Menger's Theorem is the following.

\begin{theorem}[Menger]\label{Menger}
Let $k$ be a positive integer, let $D$ be a digraph, and let $x$ and $y$ be distinct vertices in $D$ such that $xy\notin A(D)$.
Then, in $D$, either there are $k+1$ pairwise internally disjoint $(x,y)$-dipaths, or there is a $k$-separation of $(x,y)$.
\end{theorem}

For any fixed $k$, there exist algorithms running in linear time that, given a digraph $D$ and two distinct vertices $x$ and $y$ such that $xy\notin A(D)$, returns either $k+1$ internally disjoint $(x,y)$-dipaths in $D$, or a $k$-separation $(W,S,Z)$ of $(x,y)$.
Indeed, in such a particular case, any flow algorithm like Ford--Fulkerson algorithm for example, performs at most $k+1$ incrementing-path searches, because it increments the flow by 1 each time, and we stop when the flow has value $k+1$, or if we find a cut of size less than $k+1$, which corresponds to a $k$-separation.
Moreover each incrementing-path search consists in a search (usually Breadth-First Search) in an auxiliary digraph of the same size, and so is done in linear time. For more details, we refer the reader to the book of Ford and Fulkerson~\cite{FoFu62} or Chapter 7 of~\cite{BoMu08}.
We call such an algorithm a {\it Menger algorithm}.

Observe that using Menger algorithms, one can decide if there are $k$ internally disjoint $(x,y)$-dipaths in a digraph $D$.
If $xy\notin A(D)$, then we apply a Menger algorithm directly;
if $xy\in A(D)$, then we check whether there are $k-1$ internally disjoint $(x,y)$-dipaths in $D\setminus xy$.

Let $D$ be a digraph. Let $X$ and $Y$ be non-empty sets of vertices in $D$.
Two $(X,Y)$-dipaths $P$ and $Q$ are {\it disjoint} if they have no vertices in common, that is if $V(P) \cap V(Q)= \emptyset$.
A {\it $k$-separation} of $(X,Y)$ in $D$ is a partition $(W,S,Z)$ of its vertex set such that $X\subseteq W\cup S$, $Y\subseteq Z\cup S$, $|S| \leq k$, all vertices of $W$ can be reached from $X\setminus S$ by dipaths in $D[W]$,
and there is no arc from $W$ to $Z$.

Let $x$ be a vertex of $D$ and $Y$ be a non-empty subset of $V(D)\setminus \{x\}$.
Two $(x,Y)$-dipaths $P$ and $Q$ are {\it independent} if $V(P) \cap V(Q)= \{x\}$.
A {\it $k$-separation} of $(x,Y)$ in $D$ is a partition $(W,S,Z)$ of its vertex set such that $x\in W$, $Y\subseteq Z\cup S$, $|S| \leq k$, all vertices of $W$ can be reached from $x$ by dipaths in $D[W]$, and there is no arc from $W$ to $Z$.

Let $y$ be a vertex of $D$ and $X$ be a non-empty subset of $V(D)\setminus \{y\}$.
Two $(X,y)$-paths are {\it independent} if $V(P) \cap V(Q)= \{y\}$.
A {\it $k$-separation} of $(X,y)$ in $D$ is a partition $(W,S,Z)$ of its vertex set such that $W$ and $Z$ are non-empty, $X \subseteq W \cup S$, $y\in Z$, $|S| \leq k$, all vertices of $W$ can be reached from $X\setminus S$ by dipaths in $D[W]$,
and there are no arcs from $W$ to $Z$.

Let $W\subset V(D)$. The digraph $D_W$ is the one obtained from $D$ by adding a vertex $s_W$ and the arcs $s_Ww$ for all $w\in W$ and the digraph $D^W$ is the one obtained from $D$ by adding a vertex $t_W$ and the arcs $wt_W$ for all $w\in W$.

Applying Theorem~\ref{Menger} to $D_X^Y$ and $(s_X,t_Y)$ (resp.\ $D^Y$ and $(x,t_Y)$, $D_X$ and $(s_X, y)$), we obtain the following version of Menger's Theorem.

\begin{theorem}[Menger]\label{Menger-set}
Let $k$ be a positive integer, and let $D$ be a digraph. Then the following hold.
\begin{itemize}
\item[(i)] If $X$ and $Y$ are two non-empty subsets of $V(D)$, then, in $D$, either there are $k+1$ pairwise disjoint $(X,Y)$-dipaths, or there is a $k$-separation of $(X,Y)$.
\item[(ii)] If $x$ is a vertex of $D$ and $Y$ is a non-empty subset of $V(D)$, then, in $D$, either there are $k+1$ pairwise independent $(x,Y)$-dipaths in $D$, or there is a $k$-separation of $(x,Y)$.
\item[(iii)] If $X$ is a non-empty subset of $V(D)$ and $y$ is a vertex of $D$ and , then, in $D$, either there are $k+1$ pairwise independent $(X,y)$-dipaths in $D$, or there is a $k$-separation of $(X,y)$.
\end{itemize}
\end{theorem}

Moreover, a Menger Algorithm applied to $D_X^Y$ and $(s_X,t_Y)$ (resp. $D^Y$ and $(x,t_Y)$, $D_X$ and $(s_X, Y)$) finds in linear time
the $k+1$ dipaths or the separation as described in Theorem~\ref{Menger-set} (i) (resp. (ii), (iii)).

Let $x$ and $y$ be two vertices.
An {\it $(x,y)$-handle} is an $(x,y)$-dipath if $x\neq y$, and a directed cycle containing $x$ if $x=y$.
Let $y_1, \dots , y_p$ be $p$ distinct vertices, $k_1, \dots , k_p$ be positive integers and set $k=k_1+\cdots +\cdots k_p$.
One can decide if there are $k$ internally disjoint handles $P_1, \dots , P_k$ such that $k_i$ of them are $(x,y_i)$-handles, $1\leq i\leq p$, by applying a Menger algorithm between in an auxiliary digraph $D'$.  This digraph is obtained from $D-(\{y_1, \dots , y_p\}\setminus \{x\})$ as follows. Add a new vertex $y$. For each $i$, create a set $B_i$ of $k_i$ new vertices and all arcs from $N^-_D(y_i)$ to $B_i$ and from $B_i$ to $y$.

Similarly, suppose that $X$ is a set of vertices,  $y_1, \dots , y_p$ be $p$ distinct vertices not in $X$, and $k=k_1+\cdots +\cdots k_p$. One can decide if there are $k$ internally disjoint paths $P_1, \dots , P_k$, all with distinct initial vertices in $X$, and  such that $k_i$ of them are terminating in $y_i$, $1\leq i\leq p$.

\subsection{Subdivision with prescribed original vertices}

Let $F$ be a digraph and $u$ a vertex in $F$.
In an $F$-subdivision $S$, the vertex corresponding to $u$ is called the {\it $u$-vertex of $S$}. A vertex corresponding to some vertex $u\in F$ is called an {\it original} vertex.

Bang-Jensen et al.~\cite{BHM} proved that, given a digraph $D$ and a vertex $z$ in $D$, one can decide in polynomial time if $D$ contains a $W_2$-subdivision with centre $z$. Therefore $W_2$-\pb\ is polynomial-time solvable.
We now prove that we can also decide in polynomial time
if there is  a $W_2$-subdivision with two prescribed original vertices.

\begin{lemma}\label{W2}
Let $W_2$ be the $2$-wheel with centre $c$ and rim $aba$.
Given a digraph $D$ and two vertices $b'$ and $c'$, one can decide in $O(n^2(n+m))$ time if there is a $W_2$-subdivision
in $D$ with $b$-vertex $b'$ and $c$-vertex $c'$.
\end{lemma}

\begin{proof}
Let us call a $W_2$-subdivision with $b$-vertex $b'$ and $c$-vertex $c'$ a {\it $(b',c')$-forced} $W_2$-subdivision.
Let $S$ be the strong component of $b'$ in $D-c'$.
The key element is the following claim.

\begin{claim}\label{claim-W2}
$D$ contains a $(b', c')$-forced $W_2$-subdivision if and only if there exist distinct vertices $x_1$ and $x_2$ in $V(S)$ such that
there are two independent $(c',\{x_1,x_2\})$-dipaths $P_1$ and $P_2$ in $D-(S\setminus \{x_1,x_2\})$ and there are two independent $(\{x_1,x_2\},b')$-dipaths $Q_1$ and $Q_2$ in $S$.
\end{claim}

\begin{subproof}
Clearly, existence of two vertices $x_1$, $x_2$ and four dipaths $P_1,P_2,Q_1,Q_2$ as in the statement is a necessary condition for the
existence of a $(b', c')$-forced $W_2$-subdivision.
Let us now prove that it is also sufficient.
Assume that such vertices $x_1, x_2$ and dipaths $P_1,P_2,Q_1,Q_2$ exist.
Since $S$ is strong, it contains a dipath $R$ from $b'$ to
$(V(Q_1)\cup V(Q_2))\setminus \{b'\}$. (This set is not empty
since it contains $\{x_1,x_2\}\setminus \{b'\}$.)
Then $P_1 \cup P_2 \cup Q_1 \cup Q_2 \cup R$ is a
$(b', c')$-forced $W_2$-subdivision.
\end{subproof}

Our algorithm is the following. We first compute $S$, which can be done in linear time.
Then for every pair $\{x_1, x_2\}$ of vertices of $S$, we check by
running twice a Menger algorithm if the dipaths $P_1$ and $P_2$, and $Q_1$ and $Q_2$ as described in Claim~\ref{claim-W2} exist.
If yes, we return `yes', otherwise we return `no'.
The validity of this algorithm is given by Claim~\ref{claim-W2}. Since there are $O(n^2)$ pairs of vertices $\{x_1, x_2\}$, the algorithm runs in $O(n^2(n+m))$ time.
\end{proof}

A {\it spider} is a tree obtained from disjoint
directed paths by identifying one end of each path into a single vertex. This vertex is called the {\it body} of the spider.
Observe that if $T$ is a spider, then every $T$-subdivision contains $T$ as a subdigraph.
Hence a digraph contains a $T$-subdivision if and only if it contains $T$ as a subdigraph.
This implies that $T$-\pb\ can be solved in $O(n^{|T|})$ time. It also easily implies the following.

\begin{lemma}\label{lem:union-spider}
Let $F$ be a digraph and $T$ a spider.
If $F$ is tractable, then $F+T$ is also tractable.
\end{lemma}

{\it Gluing} a spider $T$ with body $b$ to $F$ at a vertex $u\in V(F)$ consists in taking the disjoint union of $F$ and $T$ and identifying $u$ and $b$.

\begin{lemma}\label{lem:add-spider}
Let $F$ be a digraph and $u$ a vertex of $F$.
If given a digraph $D$ and a vertex $v$ of $D$, one can decide in polynomial time
if there is an $F$-subdivision in $D$ such that $v$ is the $u$-vertex,
then any digraph obtained from $F$ by gluing a spider at $u$ is tractable.
\end{lemma}

\begin{proof}
Let $T$ be a spider with body $b$ and let $F'$ be the digraph obtained by gluing $T$ to $F$ at $u$.
Clearly, every $F'$-subdivision contains an $F'$-subdivision in which the arcs of $T$ are not subdivided.
Such an $F'$-subdivision is said to be {\it canonical}.

Consider the following algorithm.
For every vertex $v$ we repeat the following.
For every set $W$ of $|V(T)|-1$ vertices, we check whether $D[W\cup \{v\}]$ contains a copy of $T$ with body $v$.
This can be done in constant time.
Then we check if $D-W$ contains an $F$-subdivision with $u$-vertex $v$. This can be done in polynomial time by our assumption.

This algorithm clearly decides in polynomial time whether a given digraph $D$ contains a canonical $F'$-subdivision.
\end{proof}

A {\it $(k_1, \dots, k_p)$-spindle} is the union of $p$ pairwise internally disjoint $(a,b)$-dipaths
$P_1, \dots , P_p$ of respective lengths $k_1, \dots , k_p$. The vertex $a$ is said to be the {\it tail} of the spindle and $b$ its {\it head}.
Bang-Jensen et al.~\cite{BHM} proved that spindles are tractable. Their proof uses the following result.

\begin{theorem}[Bang-Jensen et al.~\cite{BHM}] \label{spindle}
Let $F$ be a spindle with tail $a$ and head $b$.
Given a digraph $D$ and two vertices $a'$ and $b'$, we can decide in polynomial time if $D$ contains an $F$-subdivision with $a$-vertex $a'$ and $b$-vertex $b'$.
\end{theorem}

The {\it $(k_1, \dots, k_p; l_1, \dots , l_q)$-bispindle}, denoted $B(k_1, \dots, k_p; l_1, \dots , l_q)$, is the digraph obtained from the disjoint union of a $(k_1, \dots, k_p)$-spindle
with tail $a_1$ and head $b_1$ and an $(l_1, \dots , l_q)$-spindle with tail $a_2$ and head $b_2$
by identifying $a_1$ with $b_2$ into a vertex $a$, and $a_2$ with $b_1$ into a vertex $b$.
The vertices $a$ and $b$ are called, respectively, the {\it left node} and the {\it right node} of the bispindle. The directed $(a,b)$-paths are called the {\it forward paths}, while the directed $(b,a)$-paths are called the {\it backward paths}.
Bang-Jensen et al.~\cite{BHM} proved that a bispindle is intractable
if and only if $p\ge1$, $q\ge1$ and $p+q\ge4$.
To prove that a bispindle with two forward paths and one backward path is tractable, they provided the following theorem.

\begin{theorem}[Bang-Jensen et al.~\cite{BHM}] \label{bispindle}
Let $F$ be a bispindle with two forward paths and one backward path, and let $x$ be one of its nodes.
Given a digraph $D$ and a vertex $a'$, we can decide in polynomial time if $D$ contains an $F$-subdivision with $a$-vertex $a'$.
\end{theorem}

\begin{lemma}\label{lem:sub-edge}
Let $F$ be a digraph and let $u_1, \dots , u_p$ be distinct vertices of $F$.
Suppose that for every outneighbour $v$ of $u_1$, replacing the arc $u_1v$ by a dipath $u_1wv$ of length $2$, where $w\notin V(F)$, always results in the same digraph $F'$.
Suppose that for every given digraph $D$ of order $n$ and $p$ vertices $x_1, \dots, x_p$ in $D$, one can decide in $f(n)$ time
whether there is an $F$-subdivision in $D$ such that $x_i$ is the $u_i$-vertex for every $i$.
Then given a digraph $D$ and $p$ vertices $x_1, \dots , x_p$, one can decide in $O\left({d^+(x_1)-1 \choose d^+(u_1)-1} \sum_{y\in N^+(x_1)} d^+(y) \cdot f(n-1)\right)$ time
whether there is an $F'$-subdivision in $D$ such that $x_i$ is the $u_i$-vertex for every $i$.
\end{lemma}

\begin{proof}
Set $q=d^+(u_1)$.
For every set of $q$ neighbours $y_1, \dots , y_q$ of $x_1$ and every outneighbour $z$ of $y_1$, where $z\notin\{y_2,\dots,y_q\}$, we shall give a procedure that verifies if $D$ contains an $F'$-subdivision
$S'$ such that $x_i$ is the $u_i$-vertex for all $1\leq i\leq p$, and $\{x_1y_1, \dots, x_1y_q, y_1z\}\subseteq A(S')$.
Such an $F'$-subdivision is called {\it forced}.

Let $D'$ be the digraph obtained from $D-y_1$ by deleting all arcs leaving $x_1$ except $x_1y_2, \dots , x_1y_q$, and adding the arc $x_1z$.

\begin{claim}
$D$ has a forced $F'$-subdivision if and only if $D'$ has an
$F$-subdivision such that $x_i$ is the $u_i$-vertex for every $i$.
\end{claim}

\begin{subproof}
Suppose that $S$ is an $F$-subdivision in $D'$ such that $x_i$ is the $u_i$-vertex for all $i$.
Since $x_1$ has outdegree $q$ in $D'$, we have
$\{x_1y_2, \dots , x_1y_q, x_1z\}\subseteq A(S)$.
Let $S'$ be the digraph obtained from $S$ by replacing the arc $x_1z$ by the dipath $x_1y_1z$. Because replacing the arc $u_1v$ by a dipath of length $2$ results in $F'$ for any outneighbour $v$ of $u_1$, the digraph $S'$ is an $F'$-subdivision in $D$.
Thus $S'$ is a forced $F'$-subdivision in $D$.

Conversely, assume that $S'$ is a forced $F'$-subdivision in $D$. Then the digraph $S$ obtained from $S'$ by replacing the dipath $x_1y_1z$ by the arc $x_1z$ is an $F$-subdivision in $D'$ such that $x_i$ is the $u_i$-vertex for every~$i$.
\end{subproof}

This claim implies that deciding whether $D$ contains a forced $F'$-subdivision can be done by checking whether $D'$ has an
$F$-subdivision such that $x_i$ is the $u_i$-vertex for all $i$. This can be done in $f(n-1)$ time by assumption.
By repeating this for every possible set $\{y_1, \dots , y_q, z\}$ where the $y_i$'s are distinct outneighbours of $x_1$ and $z\notin\{y_2,\dots,y_q\}$ is an outneighbour of $y_1$, we obtain an algorithm to decide whether there is an $F'$-subdivision in $D$ such that $x_i$ is the $u_i$-vertex for all $i$.
Since there are at most ${d^+(x_1)-1 \choose d^+(u_1)-1} \sum_{y\in N^+(x_1)} d^+(y)$ such sets, the running time of this algorithm is as claimed.
\end{proof}

\section{Oriented graphs of order $4$}\label{sec:oriented}

The aim of this section is to prove that every oriented graph of order $4$ is tractable.

\begin{theorem}\label{orient-4}
If $F$ is an oriented graph of order $4$, then $F$-\pb\ is polynomial-time solvable. \end{theorem}
\begin{proof}
If $F$ has no big vertices, then by Theorem~\ref{thm:DGT}, $F$-\pb\ is polynomial-time solvable.
Henceforth, we assume that $F$ has at least one big vertex. Free to consider its converse, we may assume  that $F$ has a vertex with out-degree  $3$.
Necesssarily, we must be in one the following three cases:
\begin{enumerate}
\item $|A(F)|=6$. Then $F$ is either the transitive tournament $TT_4$, or the wheel $W_3$. Bang-Jensen et al.~\cite{BHM} (Theorem 64) proved that $TT_4$-\pb\ is polynomial-time solvable.
We show in Subsection~\ref{sec:3wheel} that $W_3$ is tractable.

\begin{figure}[htb]
       \centering
       \includegraphics[height=2.8cm]{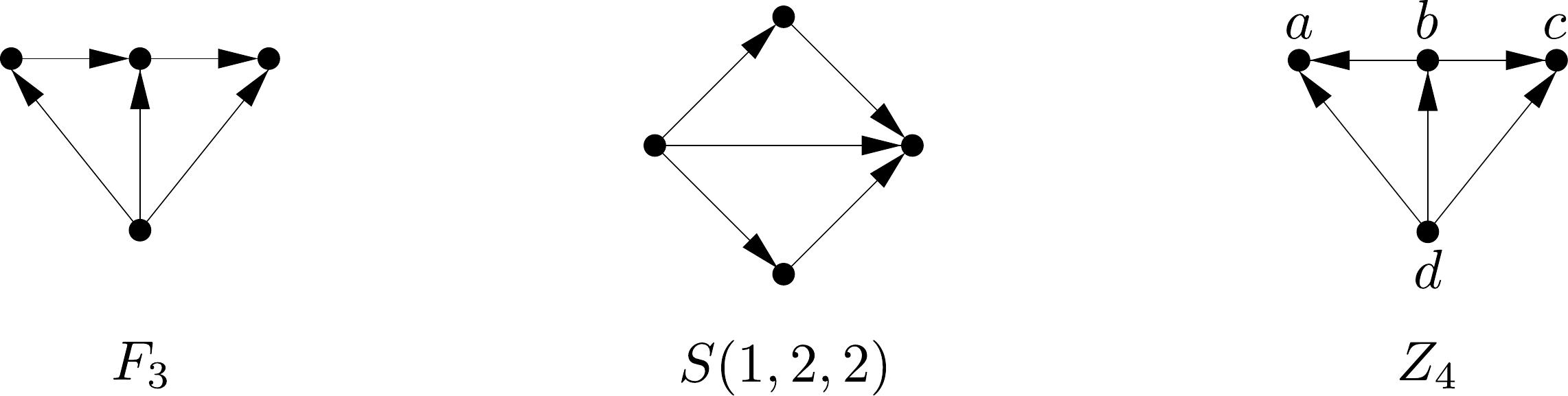}
       \caption{Oriented graphs with $4$ vertices, $5$ arcs, and a vertex of out-degree $3$}
       \label{fig:K4-e}
\end{figure}

\item $|A(F)|=5$.
Then $F$ must be one of the oriented graphs depicted Figure~\ref{fig:K4-e}.
$S(1,2,2)$ is a spindle  and $F_3$ is the $3$-fan. These digraphs have been shown to be tractable in~\cite{BHM} (Proposition 20 and Theorem 61).
We prove in Subsection~\ref{sec:Z} that $Z_4$ is tractable.

\item $|A(F)|\leq 4$. Then $F$ is either a star or a star plus an arc.  Those digraphs have been proved to be tractable in~\cite{BHM}.
\end{enumerate}
\end{proof}

\subsection{Subdivision of the $3$-wheel}\label{sec:3wheel}

\begin{theorem}\label{3wheel}
$W_3$-\pb\ can be solved in $O(n^6(n+m))$ time.
\end{theorem}

The proof of this theorem relies on the following notion.
Let $X$ be a set of three vertices.
An {\it $X$-tripod} is a digraph which is the union of a directed cycle $C$ and three disjoint dipaths $P_1, P_2, P_3$ with initial vertices in $X$ and terminal vertices in $C$. If the $P_i$ are $(X,C)$-dipaths, we
say that the tripod is {\it unfolded}. Note that the dipaths $P_i$ may be of length $0$. We shall denote the tripod described above as the $4$-tuple $(C,P_1,P_2,P_3)$.

\begin{proposition}\label{prop:unfold}
Let $X=\{x_1,x_2,x_3\}$ be a set of three distinct vertices.
Any $X$-tripod contains an unfolded $X$-tripod.
\end{proposition}


We shall consider the following decision problem.

\medskip

\noindent {\sc Tripod}\\
\underline{Input}: A strong digraph $D$ and a set $X$ of three distinct vertices of $D$.\\
\underline{Question}: Does $D$ contain an $X$-tripod?

\begin{lemma}\label{lem:tripod}
{\sc Tripod} can be solved in $O(n^2(n+m))$ time.
\end{lemma}

\begin{proof}
Let us describe a procedure {\tt tripod}$(D,X)$, solving {\sc Tripod}.

We first look for a directed cycle of length at least $3$ in $D$. This can be done in linear time.
If there is no such cycle, then we return `no'.

Otherwise we have a directed cycle $C$ of length at least $3$.
We choose a set $Y$ of three vertices in $C$ and run a Menger algorithm between $X$ and $Y$.
If such an algorithm finds three disjoint $(X,Y)$-dipaths $P_1, P_2,P_3$, then we return the tripod $(C, P_1, P_2, P_3)$.
Otherwise, the Menger algorithm finds a $2$-separation $(W,S,Z)$ of $(X,Y)$.
Note that $|S|\geq 1$ because $D$ is strong.

Assume first that $|S|=1$, say $S=\{s\}$.
Let $D_1$ be the digraph obtained from $D[W\cup S]$ by adding the arc $sw$ for every vertex $w$ in $W$ having an inneighbour $z\in Z$. We then make a recursive call to {\tt tripod}$(D_1,X)$.
This is valid by virtue of the following claim.

\begin{claim}\label{claimD1}
There is an $X$-tripod in $D$ if and only if there is an $X$-tripod in $D_1$.
\end{claim}

\begin{subproof}
Suppose first that there is an $X$-tripod in $D_1$. Then $D_1$ contains an unfolded $X$-tripod $T_1$ by Proposition~\ref{prop:unfold}.
If $T_1$ is contained in $D$, then we are done. So we may assume that it is not.
Then $T_1$ contains an arc $sw \in A(D_1)\setminus A(D)$. It can contain only one such arc since every vertex has outdegree at most one in $T_1$ and all such arcs leave $s$.
Furthermore, the head $w$ of this arc is in $W$ and $w$ has an inneighbour $z$ in $Z$.
Now, since $D$ is strong, there is an $(s,z)$-dipath $Q$ in $D$. Because there is no arc from $W$ to $Z$, all internal vertices of $Q$ are in $Z$. 
Hence the digraph $T$ obtained from $T_1$ by replacing the arc $sw$ by the dipath $Qzw$ is an $X$-tripod in $D$.

Suppose now that $D$ contains an $X$-tripod. Then it contains an unfolded $X$-tripod
$T=(C_1, P_1,\allowbreak P_2, P_3)$ by Proposition~\ref{prop:unfold}.
Since all $(X,Z)$-dipaths in $D$ go through $s$, the terminal vertices of the $P_i$ are in $W\cup S$, and $D[Z]\cap T$ is a dipath $Q$ which is a subpath of one of the $P_i$ or $C_1$.
If $Q$ is a $(t,z)$-dipath, then $T$ contains arcs $st$ and $zw$ for some $w\in W$.
Then the digraph $T_1$ obtained from $T$ by replacing $sQw$ by the arc $sw$ is an $X$-tripod in $D_1$.
\end{subproof}

Assume now that $|S|=2$, say $S=\{s_1,s_2\}$.
If there is no arc from $Z$ to $W$, let $D_2$ be the digraph obtained from $D[W\cup S]$ by adding the arc $s_1s_2$ (resp.\ $s_2s_1$) (if the arc is not already present in $D$) if there is an $(s_1,s_2)$-dipath (resp.\ $(s_2,s_1)$-dipath) in $D[Z\cup S]$.
We then make a recursive call to {\tt tripod}$(D_2,X)$.
This is valid by virtue of the following claim.

\begin{claim}\label{claimD2}
There is an $X$-tripod in $D$ if and only if there is an $X$-tripod in $D_2$.
\end{claim}

\begin{subproof}
Suppose first that there exists an $X$-tripod in $D_2$. Then there is an unfolded $X$-tripod $T_2$ in $D_2$, by Proposition~\ref{prop:unfold}.
Then either it is an $X$-tripod in $D$, or $T_2$ contains exactly one of the arcs $s_1s_2, s_2s_1$ and
this arc is not in $A(D)$. Without loss of generality, we may assume that this arc is $s_1s_2$.
Since $s_1s_2\in A(D_2)\setminus A(D)$, there is an $(s_1,s_2)$-dipath $Q$ in $D[Z\cup S]$.
Hence the digraph $T$ obtained from $T_2$ by replacing the arc $s_1s_2$ by the dipath $Q$ is an $X$-tripod in $D$.

Suppose now that $D$ contains an $X$-tripod. Then it contains an unfolded $X$-tripod
$T=(C_2, P_1, \allowbreak P_2, P_3)$ by Proposition~\ref{prop:unfold}. For $i=1,2,3$, let $y_i$ be the terminal vertex of $P_i$.
Without loss of generality, we may assume that $y_1,y_2,y_3$ appear in this order along $C_2$.
Since all $(X,Z\cup S)$-dipaths intersect $S$, one of the $y_i$, say $y_3$, must be in $W$.
The three oriented paths $P_2$, $P_1C_2[y_1,y_2]$, and $\overline{C_2}[y_3y_2]$ are independent $(W,y_2)$-paths. But the graph underlying $D$ has no edges between $W$ and $Z$, by the assumption made in the current subcase.
So $y_2$ is in $W\cup S$. Similarly, $y_1$ is in $W\cup S$.
It follows that $T\cap D[Z]$ is a dipath $Q$ which is a subpath of one of the $P_i$ or $C_2$.
Moreover, the inneighbour in $T$ of the initial vertex of $Q$ is some vertex $s\in S$ (because there is no arc from $W$ to $Z$) and the outneighbour in $T$ of the terminal vertex of $Q$ is some vertex $s'\in S$ because there is no arc from $Z$ to $W$). Furthermore $s\neq s'$ for otherwise $sQs' =C_2$ which is impossible as since $y_3\in W\cap C_2$.
Moreover, because $sQs'$ is an $(s,s')$-dipath in $D[Z\cup S]$, $ss'$ is an arc in $D_2$.
Thus the digraph $T_2$ obtained from $T$ by replacing $sQs'$ by the arc $ss'$ is an $X$-tripod in $D_2$.
\end{subproof}

Now we may assume that there is an arc $z_1w_1$ with $z_1\in Z$ and $w_1\in W$.
Since $D$ is strong, there is a cycle $C'$ containing the arc $z_1w_1$.
Necessarily, the cycle $C'$ must go through $S$ and it contains at least three vertices.

\medskip

\noindent
\underline{Case 1}: $S\subset V(C')$. Set $Y' =\{w_1, s_1,s_2\}$.
We run a Menger algorithm between $X$ and $Y'$.
If such an algorithm finds three disjoint $(X,Y')$-dipaths $P'_1,P'_2,P'_3$, then we return the $X$-tripod $(C', P'_1, P'_2, P'_3)$.

If not, we obtain a $2$-separation $(W',S',Z')$ of $(X,Y')$.
We claim that $|W'| < |W|$. Indeed, no vertex $z\in Z$ is in $W'$ because every $(X,z)$-dipath must go through $S$ and thus through $S'$. Hence $W'\subseteq W\setminus \{w_1\}$.
Now, we replace $C$ by $C'$, $Y$ by $Y'$ and $(W, S, Z)$ by $(W',S',Z')$, and then redo the procedure.

\medskip

\noindent \underline{Case 2}: $|S \cap V(C')|=1$. Without loss of generality, we may assume $S \cap V(C')=\{s_1\}$.
Set $Y' =\{w_1, s_1,z_1\}$. As in Case 1, we run a Menger algorithm between $X$ and $Y'$. If such an algorithm finds three disjoint $(X,Y')$-dipaths $P'_1,P'_2,P'_3$, then we return the $X$-tripod $(C', P'_1, P'_2, P'_3)$.

If not, the Menger algorithm returns a $2$-separation $(W',S',Z')$ for $(X, Y')$.
Observe that there is a vertex $s'_1\in S' \cap W$ because $w_1$ is reachable from $X$ in $D[W]$.
If $S'$ contains a vertex $s'_2$ in $Z$, then one can see that there are no $(X,Y')$-dipaths in $D-\{s'_1, s_2\}$. Thus, there is a $2$-separation $(W'',S'',Z'')$ of $(X, Y')$ where $S''\subseteq \{s'_1, s_2\}$ and $s_1\in Z''$.
Hence, after possibly replacing the $2$-separation $(W',S',Z')$ by $(W'', S'' , Z'')$, we may assume that $S'\subset W\cup S$.

If $|W'| < |W|$, then we set $C:=C'$, $Y:=Y'$, $(W, S, Z) :=(W',S',Z')$, and redo the procedure.

If not, then the set $R=Z \cap W'$ is not empty.
Set $L=Z\setminus R=Z\cap Z'$. There is no arc from $R$ to $L$, because $(W',S',Z')$ is a $2$-separation.
Moreover, all $(X,R)$-dipaths must go through $s_2$. In particular, $s_2\in W'$.
Let $D_3$ be the digraph obtained from $D-L$ by adding an arc $s_1w$ for every $w\in W$ having an inneighbour in $L$.
We then make a recursive call to {\tt tripod}$(D_3,X)$.
This is valid by virtue of the following claim.

\begin{claim}\label{claimD3}
There is an $X$-tripod in $D$ if and only if there is an $X$-tripod in $D_3$.
\end{claim}

\begin{subproof}
Suppose first that $D_3$ contains an $X$-tripod. Then it contains an unfolded $X$-tripod $T_3$ by Proposition~\ref{prop:unfold}.
If $T_3$ is contained in $D$, then we are done. So we may assume that $T_3$ is not contained in $D$.
Then $T_3$ contains an arc in $s_1w\in A(D_3)\setminus A(D)$. It contains only one such arc since every vertex has outdegree at most one in $T_3$ and all arcs of $A(D_3)\setminus A(D)$ leave $s_1$.
Furthermore the head $w$ of this arc is in $W$ and has an inneighbour $z\in L$.
Since $D$ is strong, there is an $(s_1,z)$-dipath $Q$ in $D$.
Moreover since $s_2\in W'$ all the $(s_2, z)$-dipaths must go through $S'$.
But $S' \subseteq W\cup \{s_1\}$, so all $(s_2, z)$-dipaths must go through $s_1$.
Thus $Q$ does not go through $s_2$. It follows that all internal vertices of $Q$ are in $Z$, because
$(W,S,Z)$ is a $2$-separation, and so in $L$ because there is no arc from $R$ to $L$.
Consequently, the digraph $T$ obtained from $T_3$ by replacing the arc $s_1w$ by the dipath $Qzw$ is an $X$-tripod in $D$.

Suppose now that $D$ contains an $X$-tripod. Then it contains an unfolded $X$-tripod
$T=(C_3, P_1,\allowbreak P_2, P_3)$ by Proposition~\ref{prop:unfold}. For $i=1,2,3$, let $y_i$ be the terminal vertex of $P_i$.
Without loss of generality, we may assume that $y_1,y_2,y_3$ appear in this order along $C_3$.
If $T$ is contained in $D-L$, then it is an $X$-tripod in $D_3$.
Hence we may assume that $T$ contains some vertices of $L$.
Observe that the arcs entering $L$ all leave $s_1$.
Hence, $y_i$ cannot be in $L$, since there are two $(X, y_i)$-dipaths in $T$, which are disjoint except for the common vertex $y_i$.
Consequently, the intersection of $T$ with $D[L]$ is a dipath $Q$ which is a subpath of one of the $P_i$ or $C_3$.
Moreover, the inneighbour in $T$ of the initial vertex of $Q$ is $s_1$ and the outneighbour in $T$ of the terminal vertex of $Q$ is some vertex $w\in W\cup \{s_1\}$, because there is no arc from $L$ to $R\cup\{s_2\}$.
But $w\neq s_1$ for otherwise $s_1Qs_1$ would be $C_3$ and would contain at most one of the $y_i$, a contradiction.
Thus the digraph $T_3$ obtained from $T$ by replacing $s_1Qw$ by the arc $s_1w$ is an $X$-tripod in $D_3$.
\end{subproof}

Claims~\ref{claimD1}, \ref{claimD2} and \ref{claimD3} ensure that our algorithm is correct. Each time we do a recursive call, the number of vertices decreases.
So we do at most $n$ of them.
Between two recursive calls, we first find a cycle of length at least $3$ in linear time, and next run a sequence of Menger algorithms to produce a new $2$-separation. At each step the size of the set $W$ decreases. Therefore, we run at most $n$ times the Menger algorithm between two recursive calls. Since a Menger algorithm runs in linear time, the time between two calls is at most $O(n(n+m))$ and so {\tt tripod} runs in $O(n^2(n+m)$ time.
\end{proof}

We now deduce Theorem~\ref{3wheel} from Lemma~\ref{lem:tripod}.

\begin{proof}[Proof of Theorem~\ref{3wheel}]
For every vertex $v$, we examine whether there is a $W_3$-subdivision with centre $v$ in $D$.
Observe that such a subdivision $S$ is the union of a directed cycle $C$, and three internally disjoint $(v,C)$-dipaths $P_1$, $P_2$, $P_3$ with distinct terminal vertices $y_1,y_2, y_3$. The cycle $C$ is contained in some strong component $\Gamma$ of $D-v$. For $i=1,2,3$, let $x_i$ be the first vertex of $P_i$ that belongs to $\Gamma$. Set $X=\{x_1,x_2,x_3\}$.
Then the paths $P_i[x_i,y_i]$, $i=1,2,3$, and $C$ form an $X$-tripod in $\Gamma$,
and the $P_i[v,x_i]$, $i=1,2,3$, are internally disjoint $(v,X)$-dipaths in $D-(\Gamma\setminus X)$.

Hence for finding a $W_3$-subdivision with centre $v$, we use the following procedure to check whether there is a set $X$ as above.
First, we compute the strong components of $D-v$.
Next, for every subset $X$ of three vertices in the same strong component $\Gamma$, we run a Menger algorithm
to check whether there are three independent $(v,X)$-dipaths
in $D-(\Gamma \setminus X)$. If yes, we check whether there is an $X$-tripod in $\Gamma$. If yes again, then we clearly have a $W_3$-subdivision with centre $v$, and we return `yes'. If not, there is no such subdivision, and we proceed to the next triple.

For each vertex $v$, there are at most $n^3$ possible triples. And for each triple we run a Menger algorithm
in time $O(n+m)$ and possibly {\tt tripod} in time $O(n^2(n+m))$. Hence the time spent on each vertex $v$ is
$O(n^5(n+m))$. As we examine at most $n$ vertices, the algorithm runs in $O(n^6(n+m))$ time.
\end{proof}

\subsection{$Z_4$-subdivision}\label{sec:Z}

In this subsection, we show that $Z_4$ is tractable. The proof relies on the following lemma.

\begin{lemma}\label{equiv-Z}
Let $D$ be a digraph.
There is a $Z_4$-subdivision in $D$ if and only if there exists four distinct vertices $a'$, $b'$, $c'$ and $d'$ in $D$ such that the following hold.
\begin{itemize}
\item[(i)] There are three independent $(d',\{a',b',c'\})$-dipaths.
\item[(ii)] There are two independent $(b', \{a',c'\})$-dipaths.
\end{itemize}
\end{lemma}

\begin{proof}
If $D$ contains a $Z_4$-subdivision $S$, then the vertices $a', b',c',d'$ corresponding to $a,b,c,d$ (as indicated on Figure~\ref{fig:K4-e}) clearly satisfy conditions (i) and (ii).

Conversely, suppose that $D$ contains four vertices $a', b',c',d'$ satisfying conditions (i) and (ii).
Let $P_1, P_2, P_3$ be three independent $(d',\{a',b',c'\})$-dipaths with $t(P_1)=a'$, $t(P_2)=b'$ and $t(P_3)=c'$;
let $Q_1,Q_2$ be two independent $(b', \{a',c'\})$-dipaths with $t(Q_1)=a'$ and $t(Q_2)=c'$.

We consider such vertices $a',b',c',d'$ and dipaths such that the sum of the lengths of $P_1$, $P_2$, $P_3$, $Q_1$ and
$Q_2$ is minimized.

\begin{claim}\label{Z1} $V(Q_1)\cap V(P_1)=\{a'\}$ and $V(Q_2)\cap V(P_3)=\{c'\}$.
\end{claim}

\begin{subproof}
Suppose $V(Q_1)\cap V(P_1)\neq \{a'\}$. Then there is a vertex $a''$ distinct from $a'$ in $V(Q_1)\cap V(P_1)$.
The vertices $a'', b',c', d'$ satisfy condition (i) with $P_1[d',a'']$, $P_2$, $P_3$
and condition (ii) with $Q_1[b',a'']$, $Q_2$. This contradicts our choice of $a',b',c',d'$ and the corresponding paths, and so $V(Q_1)\cap V(P_1)=\{a'\}$.

The conclusion that $V(Q_2)\cap V(P_3)=\{c'\}$ is proved in the same way; the details are omitted.
\end{subproof}

\begin{claim}\label{Z2}
$(V(Q_1)\cup V(Q_2))\cap V(P_2)=\{b'\}$.
\end{claim}

\begin{subproof}
Suppose not. Then let $b''$ be the last vertex distinct from $b'$ along $P_2$ which is in
$V(Q_1)\cup V(Q_2)$. By symmetry, we may assume that $b''\in V(Q_1)$.
But the four vertices $a', b'',c', d'$ satisfy condition (i) with $P_1$, $P_2[d',b'']$, $P_3$
and condition (ii) with $Q_1[b'',a']$, $P_2[b'',b']Q_2$. This contradicts our choice of $a',b',c',d'$ and proves our claim.
\end{subproof}

\begin{claim}\label{Z3}
$V(Q_1)\cap V(P_3)=\emptyset$ and $V(Q_2)\cap V(P_1)=\emptyset$.
\end{claim}

\begin{subproof}
Suppose not.
Then $V(Q_1)\cap V(P_3)$ or $V(Q_2)\cap V(P_1)$ is not empty.

Assume first that these two sets are both non-empty.
Let $a''$ be a vertex in $V(Q_2)\cap V(P_1)$ and $c''$ be a vertex in $V(Q_1)\cap V(P_3)$.
Then the four vertices $a'', b',c'', d'$ satisfy condition (i) with $P_1[d',a'']$, $P_3[d',c'']$, $P_2$
and condition (ii) with $Q_2[b',a'']$, $Q_1[b',c'']$. This contradicts our choice of $a',b',c',d'$.

Hence, exactly one of the two sets is empty. By symmetry, we may assume that $V(Q_1)\cap V(P_3)\neq \emptyset$,
Let $b''$ be a vertex in $V(Q_1)\cap V(P_3)$.
Now the four vertices $a', b'',c', d'$ satisfy condition (i) with $P_1$, $P_3[d',b'']$, $P_2Q_2$
and condition (ii) with $Q_1[b'',a']$, $P_3[b'',c']$. This contradicts our choice of $a',b',c',d'$ and proves our claim.
\end{subproof}

Claims~\ref{Z1}, \ref{Z2} and \ref{Z3} imply that
$P_1\cup P_2\cup P_3\cup Q_1\cup Q_2$ is a $Z_4$-subdivision.
\end{proof}

\begin{theorem}\label{poly-Z}
$Z_4$-\pb\ can be solved in $O(n^4(n+m))$ time.
\end{theorem}

\begin{proof}
By Lemma~\ref{equiv-Z}, $Z_4$-\pb\ is equivalent to deciding whether there are four vertices
satisfying the condition (i) and (ii) of the lemma.
But given four vertices $a', b',c',d'$, one can check in linear time if conditions (i) and (ii) hold by running two Menger algorithms.
Since there are $O(n^4)$ sets of four vertices in $D$, $Z_4$-\pb\ can be solved in $O(n^4(n+m))$ time.
\end{proof}

\section{Some tractable digraphs}\label{sec:easy}

\subsection{Easier cases}\label{subsec:easier}

A {\it symmetric star} is a symmetric digraph associated to a star. The {\it centre} of a symmetric star is the centre of the star to which it is associated.
A {\it superstar} is a digraph obtained from a symmetric star by adding an arc joining two non-central vertices. The {\it centre} of a superstar is the centre of the star from which it is derived.
The symmetric star of order $k+1$ is denoted by $SS_k$ and the superstar of order $k+1$ is denoted by $SS^*_k$.
An $SS_k$-subdivision with centre $a$ is the union of $k$ internally disjoint $(a,a)$-handles. Therefore, on can decide
if there is an $SS_k$-subdivision with centre $a$ in linear time using a Menger algorithm.
Bang-Jensen et al.~\cite{BHM} showed that $SS^*_3$-\pb\ is polynomial-time solvable.
Their result can be extended to all superstars.

\begin{theorem}\label{superstar}
Let $k$ be a positive integer.
Given digraph $D$ and a vertex $v$ of $D$, on can decide in $O(n^{2k}(n+m))$-time whether $D$ contains an $SS^*_k$-subdivision with centre $v$.
\end{theorem}

\begin{proof}
We describe a procedure that given $v$, a set $X=\{x_1, \dots , x_k\}$ of $k$ distinct outneighbours of $v$ and a set $Y=\{y_1, \dots , y_k\}$ of
$k$ distinct inneighbours of $v$ checks if there is an
$SS^*_k$-subdivision $S$ with centre $v$ such that $\{vx_1, \dots, vx_k\} \cup \{y_1v, \dots, y_kv\} \in A(S)$.
(Note that it is allowed that $X\cap Y\neq \emptyset$.)
Such a subdivision will be called {\it $(v, X, Y)$-forced}.

Applying a Menger algorithm, check whether in $D-v$ there are $k$ disjoint dipaths $P_1, \dots , P_k$ from
$X$ to $Y$.
If not, then $D$ certainly does not contain any $(v, X, Y)$-forced $SS^*_k$-subdivision.
If yes, then check whether there is a dipath $Q$ from some $P_i$ to a different $P_j$ whose internal vertices are not in $\{v\}\cup \bigcup_{i=1}^k P_i$.
This can be done in linear time by running a search on the digraph obtained from $D-v$ by contracting each path $P_i$ into a single vertex.
If such a dipath $Q$ exists, then $P_1, \dots , P_k$ and $Q$ together with $v$ and the arcs from $v$ to $X$ and from $Y$ to $v$ form a $(v, X, Y)$-forced $SS^*_k$-subdivision.
If not, then no $(v, X, Y)$-forced $SS^*_k$-subdivision using the chosen arcs exists, because there is no vertex $x\in X$ with two vertices of $Y$ in its outsection in $D-v$.

Applying this linear-time procedure for every possible pair $(X, Y)$, we can decide in $O(n^{2k}(n+m))$-time whether $D$ contains an $SS^*_k$-subdivision with centre $v$.
\end{proof}

\begin{corollary}\label{cor:superstar}
For every positive integer $k$, $SS_k^*$-\pb\ can be solved in $O(n^{2k+1}(n+m))$-time.
\end{corollary}

\begin{proposition}\label{poly}
For $1\leq i\leq 8$,
the digraph $E_i$ depicted in Figure~\ref{fig:easy} is tractable.
\end{proposition}

\begin{proof}
$i=1$:
Let us describe a procedure that, given two distinct vertices $a'$ and $d'$ in $D$ and two outneighbours $s_1$, $s_2$ of $a'$ distinct from $d'$, decides whether there is an $E_1$-subdivision with $a$-vertex $a'$ and $d$-vertex $d'$ such that $a's_1$ and $a's_2$ are arcs of $S$. Such a subdivision is said to be {\it $(a's_1,a's_2,d)$-forced}.

We check whether there is a dipath $Q$ from $\{s_1,s_2\}$ to $d'$ in $D-a'$, and with a Menger algorithm we check whether there are two independent $(\{s_1,s_2\},a')$-dipaths $P_1$ and $P_2$ in $D-d'$.
If these three dipaths do not exist, then $D$ contains no $(a's_1,a's_2,d)$-forced $E_1$-subdivision, and we return `no'.
If the three paths $Q,P_1,P_2$ exist, then we return `yes'. Indeed, denoting by $c'$ the last vertex along $Q$
in $P_1\cup P_2$, the digraph $a's_1\cup P_1 \cup a's_2 \cup P_2\cup Q[c',d']$ is an $(a's_1,a's_2,d)$-forced $E_1$-subdivision.

Applying the above procedure for all possible triples
$(a's_1, a's_2,d')$, one solves $E_1$-\pb\ in $O(n^4(n+m))$ time.

\medskip

\noindent
$i=2$:
Let us describe a procedure that given two distinct vertices $a'$ and $d'$ in $D$, a set $U=\{u_1,u_2,u_3\}$ of three outneighbours of $a'$, returns `yes' if it finds an $E_2$-subdivision and returns `no' only if there is no $E_2$-subdivision with $a$-vertex $a'$ and $d$-vertex $d'$ such that $\{a'u_1, a'u_2, a'u_3\}\subseteq A(S)$. Such a subdivision is said to be {\it $(a',d', U)$-forced}.

We check with a Menger algorithm whether $|S_{D-a'}^-(d')\cap U|\geq 2$ and whether there are
three internally disjoint dipaths $P_1, P_2, P_3$ with distinct initial vertices in $U$ and with $t(P_1)=t(P_2)=a'$ and
$t(P3)=d'$.  If these two conditions are not both fulfilled, then $D$ contains no $(a',d',U)$-forced $E_2$-subdivision, and we return `no'.
If these conditions are fulfilled, then we return `yes'.
Indeed consider three dipaths $P_1,P_2,P_3$ as above. Without loss of generality, $s(P_i)=u_i$ for $1\leq i\leq 3$.
Since $|S_{D-a'}^-(d')\cap U|\geq 2$, there exists a $(P_1\cup P_2, P_3)$-dipath in $D-a'$. Let us denote its terminal vertex by $d''$.
Then the union of the directed cycles $a'u_1P_1$, $a'u_2P_2$, and the dipaths $a'u_3P_3[u_3,d'']$, and $Q$ is an $E_2$-subdivision.

Applying the above procedure for all possible triples
$(a',d', U)$, one solves $E_2$-\pb\ in $O(n^5(n+m))$ time.

\medskip

\noindent
$i=3$:
Let us describe a procedure that given two distinct vertices $a'$ and $d'$ in $D$ and two outneighbours $s_1$, $s_2$ of $a'$ distinct from $d'$, returns `yes' when it finds an $E_3$-subdivision and returns `no' only if there is no $E_3$-subdivision with $a$-vertex $a'$ and $d$-vertex $d'$ such that $\{a's_1, a's_2\}\subseteq S$. Such a subdivision is said to be {\it $(a's_1,a's_2,d')$-forced}.

We check whether there is an $(\{s_1,s_2\},d')$ dipath $Q$ in $D-a'$ and whether there are three independent $(\{s_1,s_2,d'\},a')$-dipaths $P_1,P_2,P_3$ in $D$.
If these two conditions are not both fulfilled, then $D$ contains no $(a's_1,a's_2,d')$-forced $E_3$-subdivision, and we return `no'.
If these conditions are fulfilled then we return `yes'.

Indeed, suppose there are four such dipaths $Q,P_1,P_2, P_3$. We may assume without loss of generality that $s(P_3)=d'$. Denote by $c'$ the last vertex along $Q$ in $P_1\cup P_2$, and by $d''$ the first vertex in $Q[c',d']$ which is on $P_3$.
Then the union of the two directed cycles $a's_1P_1a', a's_2P_2a'$ and the dipaths $Q[c',d'']$ and $P_3[d'',a']$ is
an $E_3$-subdivision.

Applying the above procedure for all possible triples
$(a's_1, a's_2, d')$, one solves $E_3$-\pb\ in $O(n^4(n+m))$ time.

\medskip

\noindent
$i=4$:
Let us describe a procedure that, given an arc $sa'$ and a vertex $d'\notin \{s,a'\}$, checks whether there is an $E_4$-subdivision $S$ with $a$-vertex $a'$, $d$-vertex $d'$, and such that $sa'\in A(S)$. Such a subdivision is said to be {\it $(sa', d')$-forced}.

We check with a Menger algorithm whether there are three independent $(a',\{s,d'\})$-dipaths, where two of the paths end up at $d'$ and one at $s$. If three such dipaths do not exist, then there is clearly no $(sa',d')$-forced $E_4$-subdivision, and we return `no'. If three such dipaths exist, then their union together with the arcs $sa'$ form an $(sa',d')$-forced $E_4$-subdivision.

Applying the above procedure for all possible pairs $(sa',d')$, one solves $E_4$-\pb\ in $O(mn(n+m))$ time.

\medskip

\noindent
$i=5$:
Let us describe a procedure that, given two distinct vertices $b',c'$ and a set $S=\{s_1,s_2,s_3\}$ of three distinct inneighbours of $b'$ checks whether there is an $E_5$-subdivision $S'$ with $b$-vertex $b'$, $c$-vertex $c'$, and such that $\{s_1b', s_2b', s_3b'\}\subset A(S')$. Such a subdivision is said to be {\it $(b', c', S)$-forced}.

We check with a Menger algorithm, if there are three independent $(c',S)$-dipaths $P_1, P_2, P_3$, and we check whether there is a $(b',S\setminus\{c'\})$-dipath $Q$ in $D-c'$. If four such dipaths do not exist, then we return `no' because there is no $(b', c', S)$-forced $E_5$-subdivision.
If such dipaths $P_1, P_2, P_3$ and $Q$ exist, then
let $x$ be the first vertex of $Q$ in $P_1\cup P_2\cup P_3$.
Then the union of $P_1, P_2, P_3, Q[b',x]$ and the three arcs $s_1b', s_2b', s_3b'$ form a $(b', c', S)$-forced $E_5$-subdivision.

Applying the above procedure for all possible triples $(a',b', S)$, one solves $E_5$-\pb\ in $O(n^5(n+m))$ time.

\medskip

\noindent
$i=6$:
Observe that every $E_6$-subdivision may be seen as an $E_6$-subdivision in which the arc $dc$ is not subdivided.
Henceforth, by an $E_6$-subdivision, we mean such a subdivision.

Let us describe a procedure that, given two disjoint arcs, $sb'$ and $d'c'$, returns `yes' if it finds an $E_6$-subdivision and returns `no' only if there is no $E_6$-subdivision $S$ with $b$-vertex $b'$, $c$-vertex $c'$, $d$-vertex $d'$ and such that $\{sb', d'c'\}\subseteq A(S)$. Such a subdivision is called {\it $(sb', d'c')$-forced}.

Applying a Menger algorithm, we check whether in $D$ there are three independent $(b', \{s,c',d'\})$-dipaths $P_1, P_2, P_3$ with $t(P_1)=s$ and applying a search we check whether there is a $(c',s)$-dipath $Q$ in $D-\{b',d'\}$.
Clearly, if four such dipaths do not exist, then $D$ contains no $(sb', d'c')$-forced $E_6$-subdivision, so we return `no'.
Conversely, if these dipaths exist, then $Q$ contains a $(c',P_1)$-subdipath $R$.
Let $c''$ be the last vertex along $R$ in $V(P_2\cup P_3)$.
Now in $P_2\cup P_3\cup R[c', c'']\cup d'c'$, there are two internally disjoint $(b',c'')$-dipaths $P'_2, P'_3$.
Thus $P_1\cup sb' \cup P'_2\cup P'_3 \cup R[c'', t(R)]$ is an $E_6$-subdivision, and we return `yes'.

Doing this for every possible pair $(sb', d'c')$, one decides in $O(m^2(n+m))$ time whether $D$ contains an $E_6$-subdivision.

\medskip

\noindent
$i=7$:
We proceed in two stages. We first check whether there is an $E_7$-subdivision in which the arc $ab$ is not subdivided.
Next we check whether there is an $E_7$-subdivision in which the arc $ab$ is subdivided.

In the first stage we decide whether there is an $E_7$-subdivision with $a$-vertex $a'$ and $b$-vertex $b'$ for some arc $a'b'$. To do so, for every dipath $a'uv$ in $D-b'$, we check whether there is
an $E_7$-subdivision with $a$-vertex $a'$ and $b$-vertex $b'$, and which contains the arcs of $\{a'u, uv, a'b'\}$. Such a subdivision is said to be {\it $(a'uv, a'b')$-forced}.

We proceed as follows.
Applying a Menger algorithm, we check whether in $D-u$ there are independent $(\{v,b'\},a')$-dipaths $P_1$ and $P_2$ with $s(P_1)=v$, and applying a search we check whether there is a $(v,b')$-dipath $Q$ in $D-a'-u$.
Clearly, if three such dipaths do not exist, then $D$ contains no $(a'uv, a'b')$-forced $E_7$-subdivision, so we return `no'.
Conversely, if these dipaths exists, then $Q$ contains a $(P_1,P_2)$-subdipath $R$.
Then the union of $P_1$, $P_2$, $R$, $a'uv$, and $a'b'$ is an $E_7$-subdivision, and we return `yes'.
Doing this for every possible pair $(a'uv, a'b')$, one decides in $O(m^2(n+m))$ time that either $D$ contains an $E_7$-subdivision, or that $D$ contains no $E_7$-subdivision in which the arc $ab$ is not subdivided.

\smallskip

Let $G_7$ be the digraph obtained from $E_7$ by subdividing the arc $ab$ into a dipath $awb$ of length $2$.
The second stage consists in deciding whether $D$ contains an $G_7$-subdivision.
We use a procedure similar to the one for detecting superstar subdivision.
Given a pair $\{a'w_1x_1, a'w_2x_2\}$ of dipaths that are disjoint except for their initial vertex $a'$, and two distinct inneighbours $y_1,y_2$ of $a'$
that are not in $\{w_1, w_2\}$ (allowing the possibility that $\{x_1,x_2\}\cap \{y_1,y_2\}\ne \emptyset$), the procedure returns `yes' if it finds an $G_7$-subdivision and returns `no' only if there is
no $G_7$-subdivision with $a$-vertex $a'$ containing all arcs in $A'=\{a'w_1, w_1x_1, a'w_2, w_2x_2, y_1a', y_2a'\}$. Such a subdivision is called {\it $A'$-forced}.

The procedure proceeds as follows.
With a Menger algorithm, we first check whether in $D-\{a',w_1,w_2\}$ there are two disjoint dipaths $P_1, P_2$ from
$\{x_1,x_2\}$ to $\{y_1,y_2\}$.
If not, then $D$ certainly does not contain any $A'$-forced $G_7$-subdivision.
If yes, then check whether there is a $(P_1,P_2)$-dipath $Q$ in $D-\{a',w_1,w_2\}$.
If such a dipath exists, then the union of the paths $P_1, P_2, Q, a'w_1x_1, a'w_2x_2$ and the arcs $y_1a'$ and $y_2a'$
is an $G_7$-subdivision and we return `yes'.
Next, we check if there is a $(P_2,P_1)$-dipath $Q$ in $D-\{a',w_1,w_2\}$.
If $Q$ exists, we return `yes'.
If not, then no $A'$-forced $G_7$-subdivision exists, because there is no vertex $x\in \{x_1,x_2\}$ with two vertices of $\{y_1,y_2\}$ in its outsection in $D-\{a',w_1,w_2\}$. So we return `no'.

This procedure runs in linear time. Thus, running it for
every possible set $A'$, one decides in $O(m^2n^3(n+m))$ time whether $D$ contains an $G_7$-subdivision, which is nothing but an $E_7$-subdivision in which the arc $ab$ is subdivided.

Doing the two stages one after another, we obtain an
$O(m^2n^3(n+m))$-time algorithm for solving $E_7$-\pb.

\medskip

\noindent
$i=8$:
Similarly to the case $i=7$, we proceed in two stages. We first check whether there is an $E_8$-subdivision in which the arc $ab$ is not subdivided.
Next we check whether there is an $E_8$-subdivision in which the arc $ab$ is subdivided.

The first stage is the following.
For every vertex $a'$, every two distinct outneighbours $b', u$, and every inneighbour $t'$ of $a'$ distinct from $b'$ and $u$, we run a procedure that returns `yes' if it finds an $E_8$-subdivision, and return `no' if there is no
$E_8$-subdivision with $a$-vertex $a'$ and $b$-vertex $b'$ and whose arc set includes $\{t'a', a'b',a'u\}$. Such a subdivision is called {\it $(t'a', a'b', a'u)$-forced}.
The procedure is the following.
With a Menger algorithm, we check whether in $D-u$ there are two independent $(b',\{a',t'\})$-dipaths $P_1, P_2$ and whether there is a $(u,t')$-dipath $Q$ in $D-\{a',b'\}$.
If three such paths do not exist, then $D$ certainly contains no $(t'a', a'b', a'u)$-forced $E_8$-subdivision and we return `no'.
If these three paths exist, we then we return `yes'. Indeed let $d'$ be the first vertex along $Q$ in $P_1\cup P_2$.
Now the union of $P_1$, $P_2$, $Q[u,d']$, $a'b'$, $t'a'$ and $a'u$ is an $E_8$-subdivision with $a$-vertex $a'$ and $b$-vertex $b'$.

Doing this for every possible triple $(t'a', a'b', a'u)$,
one can decide in time $O(n^2m(n+m))$ whether there is an $E_8$-subdivision with in which the arc $ab$ is not subdivided.

Observe that an $E_8$-subdivision in which $ab$ is subdivided is an $G_7$-subdivision.
Hence the second phase is exactly the same as the one for $E_7$.

Doing the two stages one after another, we obtain an
$O(m^2n^3(n+m))$-time algorithm for solving $E_8$-\pb.
\end{proof}


\subsection{$E_{9}$ is tractable}

\begin{theorem}\label{E_9}
$E_9$-{\sc Subdivision} can be solved in $O(n^7 (n+m))$ time.
\end{theorem}

The proof relies on the following notion.
A \emph{shunt} is a digraph composed of three dipaths $P$, $Q$ and $R$ such that $R$ has length at least $2$, $s(R) \in P$, $t(R) \in Q$ and $P, Q, R^0$ are disjoint. We frequently refer to a shunt by the triple $(P, Q, R)$.
An {\it $(S,T)$-shunt} is a shunt  $(P, Q, R)$ such that $\{s(P), s(Q)\} = S$ and $\{t(P), t(Q)\} = T$.

We consider the following decision problem.

\medskip

\noindent {\sc shunt}\\
\underline{Input}: A digraph $D$ and four distinct vertices $s_1, s_2, t_1, t_2$.\\
\underline{Question}: Does $D$ contain an $(\{s_1,s_2\}, \{t_1,t_2\})$-shunt?\\

Assume that there are two disjoint dipaths $P, Q$ from $\{s_1, s_2\}$ to $\{t_1, t_2\}$ in $D$.
We now give some necessary and sufficient conditions considering $P$ and $Q$ for $D$ to have an $(\{s_1,s_2\}, \{t_1,t_2\})$-shunt.

For any vertex $x$ in $V(P)$, an {\it $x$-bypass} is a dipath $B$ internally disjoint from $P$ and $Q$ with initial vertex in $P[s(P),x[$ and terminal vertex in $P]x, t(P)]$.
Similarly, for any vertex $x$ in $V(Q)$, an {\it $x$-bypass} is a dipath $B$ internally disjoint from $P$ and $Q$ with initial vertex in $Q[s(Q),x[$ and terminal vertex in $Q]x, t(Q)]$.
If $x$ is the end-vertex of an arc between $P$ and $Q$, then every $x$-bypass is said to be an {\it arc bypass} (figure \ref{fig: bypass}).
 A {\it crossing} (with respect to $P$ and $Q$) is a pair of arcs $\{uv,u'v'\}$ such that $u$ is before $v'$ along $P$ and $u'$ is before $v$ along $Q$.
If $uv'$ is an arc of $P$ and $u'v$ is an arc of $Q$, then the crossing is {\it tight}. Otherwise it is {\it loose}.

 Let $C=\{uv, u'v'\}$ be a tight crossing.
A {\it $C$-forward path} is a dipath internally disjoint from $P$ and $Q$ either with initial vertex $u$ and terminal vertex $v'$, or  with initial vertex $u'$ and terminal vertex $v$ (figure \ref{fig: crossing}). A {\it $C$-backward path} is a dipath internally disjoint from $P$ and $Q$ either with initial vertex in $P[v',t(P)]$ and terminal vertex in $P[s(P), u]$, or  with initial vertex in $Q[v,t(Q)]$ and terminal vertex in $Q[s(Q), u']$.
A {\it $C$-backward arc} 
is an arc that forms a $C$-backward 
path of length $1$.
A {\it $C$-bypass} is an $x$-bypass $B$, where $x$ is an endvertex of a $C$-backward arc and if $x \in P[s(P),u]$ (resp. $Q[s(Q), u']$), $t(B)$ is also in $P[s(P),u]$ (resp. $Q[s(Q), u']$), or if  $x \in P[v', t(P)]$ (resp. $Q[v, t(Q)]$), $s(B)$ is also in $P[v', t(P)]$ (resp. $Q[v, t(Q)]$) (figure \ref{fig: backward}).

\begin{figure*}[!htb]
   \centerline{
   \subfigure[Arc bypass $B$]{\includegraphics[width=4.5cm]{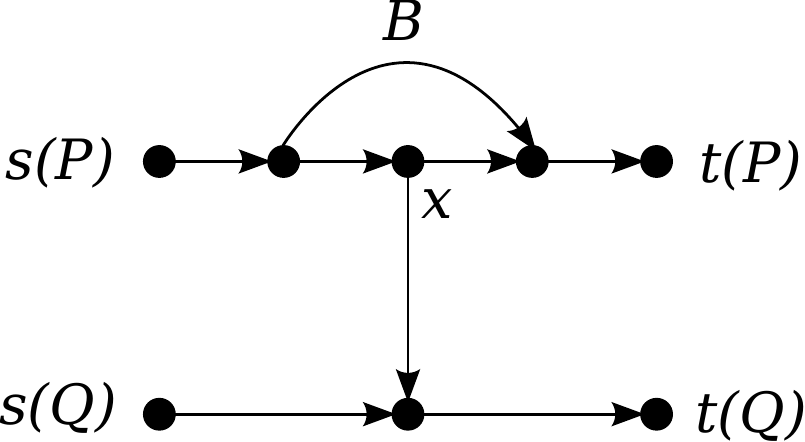}
   \label{fig: bypass}}
   \hfil
   \subfigure[A tight crossing C and a C-forward path in bold]{\includegraphics[width=5.0cm]{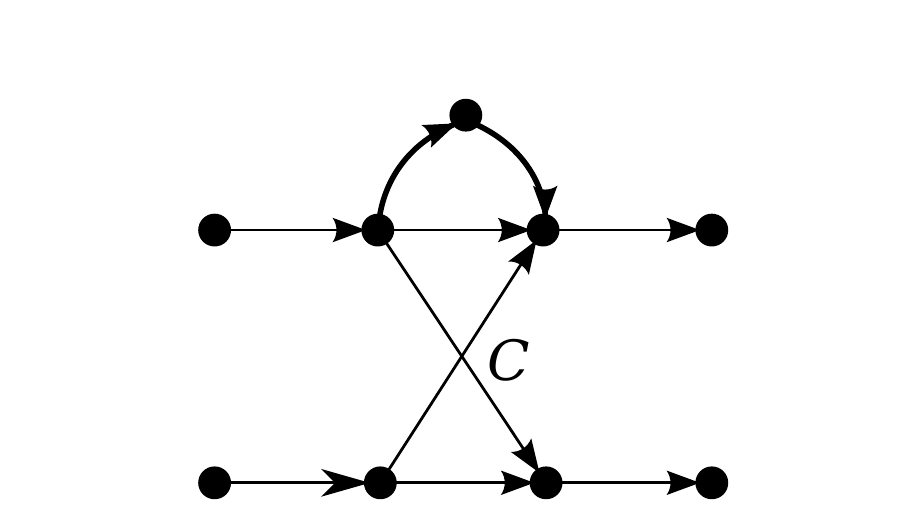}
   \label{fig: crossing}}
   \hfil
   \subfigure[A C-backward arc in bold and a $C$-bypass $B$]{\includegraphics[width=5.0cm]{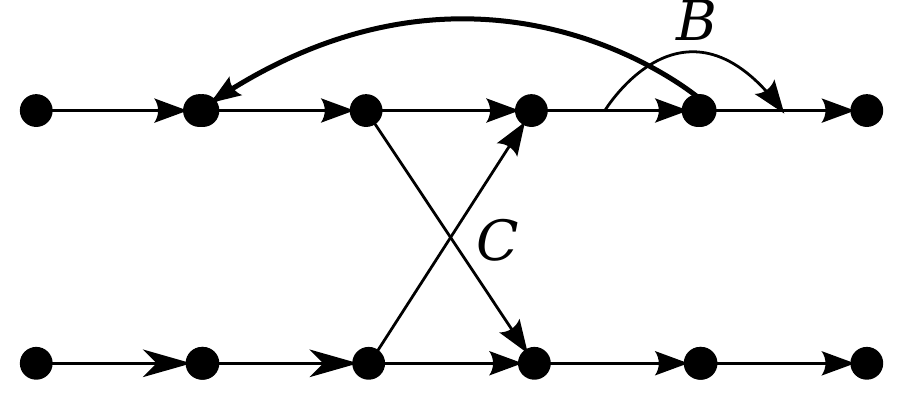}
   \label{fig: backward}}
   }
   \caption[Conditions considering $P$ and $Q$ for $D$ to have an $(\{s_1,s_2\}, \{t_1,t_2\})$-shunt.]{Conditions considering $P$ and $Q$ for $D$ to have an $(\{s_1,s_2\}, \{t_1,t_2\})$-shunt.}
\end{figure*}

 \begin{lemma}\label{lem:shunt-equiv}
  Let $D$ be a digraph, and let $P$ and $Q$ be two disjoint dipaths from $\{s_1,s_2\}$ to $\{t_1, t_2\}$.
  $D$ has an $(\{s_1,s_2\}, \{t_1,t_2\})$-shunt if and only if one of the following holds~:
 \begin{itemize}
 \item[(a)] there is a $(P,Q)$-dipath or a $(Q,P)$-dipath $R$ of length $\ge 2$;
 \item[(b)] there is an arc bypass for some arc $uv$ between $P$ and $Q$;
 \item[(c)] there is a loose crossing;
 \item[(d)] there is a tight crossing $C$ with a $C$-forward path,  a $C$-backward path of length at least $2$ or a crossing bypass.
 \end{itemize}
\end{lemma}
 \begin{proof}
Let us first price that if one of (a)--(d) holds, then $D$ has an $(\{s_1,s_2\}, \{t_1,t_2\})$-shunt.
 \begin{itemize}
 \item[(a)] If such a dipath $R$ exists, then $(P, Q, R)$ or $(Q,P,R)$ is an $(\{s_1,s_2\}, \{t_1,t_2\})$-shunt.

 \item[(b)] If $B$ is a $u$-bypass and $u \in V(P)$, then  $(P[s(P), s(B)] \cup B \cup P[t(B), t(P)], Q, P[s(B), u] \cup uv)$ is an $(\{s_1,s_2\}, \{t_1,t_2\})$-shunt. There is a shunt constructed analogously if $u\in V(Q)$ and also when $B$ is a $v$-bypass.

 \item[(c)] Let $\{uv,u'v'\}$ be a loose crossing. By symmetry, we may assume that $uv'$ is not an arc.
Then $(P[s(P),u]\cup uv \cup Q[v, t(Q)], Q[s(Q),u']\cup u'v' \cup P[v', t(P)], P[u, v'])$ is an $(\{s_1,s_2\}, \{t_1,t_2\})$-shunt.

 \item[(d)] Let $C=\{uv, u'v'\}$ be a tight crossing.

If there is a $C$-forward path, then  replacing the arc $uv'$ on $P$ or the arc $u'v$ on $Q$ by this $C$-forward path, we obtain two dipaths with a loose crossing,
so we are done by (c).

If there is a $C$-backward path $R$ of length at least $2$,
then $P[s(P),u]\cup uv\cup Q[v, t(Q)]$, $Q[s(Q),u']\cup u'v' \cup P[v', t(P)])$ and $R$ form an $(\{s_1,s_2\}, \{t_1,t_2\})$-shunt.

Suppose now that $B$ is a $C$-bypass.
By symmetry and directional duality, we may assume that $B$ is an $x$-bypass with $t(B)\in P[s(P),u]$.
Let $a=xw$ be the corresponding
$C$-backward arc $a$, where $w \in P[v', t(P)]$.
Then $(Q[s(Q),u']\cup u'v' \cup P[v', t(P)]), P[s(P), s(B)]\cup B \cup P[t(B),u]\cup uv \cup Q[v,t(Q)], wx \cup P[x,t(B)])$ is an $(\{s_1,s_2\}, \{t_1,t_2\})$-shunt.

 \end{itemize}

 \medskip

 Let us now prove the reciprocal by the contrapositive.
Suppose for a contradiction none of (a)--(d) holds, but $D$ contains an $(\{s_1,s_2\}, \{t_1,t_2\})$-shunt $(P',Q',R')$.
Without loss of generality, we may assume that this shunt maximizes $|(A(P)\cup A(Q))\cap (A(P')\cup A(Q'))|$.
Free to swap the names of $P$ and $Q$, we may assume that $s(P)=s(P')$.

Let $u$ be the farthest vertex along $P'$ such that $P'[s(P'), u]$ does not intersect $Q$.
Necessarily $u\in V(P)$ for otherwise there would be a dipath of length at least $2$ from $P$ to $Q$.
In addition, for the same reason, if $u\neq t(P)$, then the outneighbour $v$ of $u$ in $P'$ must be in $Q$.
Hence all vertices of  $P'[s(P'), u]\cap P$ are in $P[s(P), u]$, for otherwise there would be a $u$-bypass in $P$, which would be an arc bypass for $uv$.
Note also that, for every vertex $x$ in $P[s(P),u]-P'$, there is a subdipath of $P'$ which is an $x$-bypass.
So $Q' \cap P[s(P),u]=\emptyset$, for otherwise in $Q'$ there would be a dipath from $Q$ to $P[s(P),u]$ which is either has length at least $2$ or is an arc with an arc bypass in $P'$.
Let $R''$ be the shortest subdipath of $P'$ with initial vertex in $V(P)$ and terminal vertex $s(R')$ if $s(R')\in P'[s(P'), u]$,
and let $R''$ be the path of length $0$ $(s(R'))$ otherwise.
Now, $(P'',Q'', R)=(P[s(P), u]\cup P'[u, t(P')], Q', R''\cup R')$ is an  $(\{s_1,s_2\}, \{t_1,t_2\})$-shunt.
Moreover if $P'[s(P'), u] \neq P[s(P), u]$, then  $P''$ and $Q''$ have more arcs in common with $P$ and $Q$ than $P'$ and $Q'$, which contradicts our choice of $(P', Q', R')$.
Therefore  $P'[s(P'), u]  = P[s(P), u]$.

Let $u'$ be the farthest vertex along $Q'$ such that $Q'[s(Q'), u']$ does not intersect $P$.
As above, one shows that $Q'[s(Q'), u'] = Q[s(Q), u']$.

If $u=t(P)$, then $P'=P$ and necessarily $Q = Q'$. Thus $R'$ is a dipath of length at least $2$ from $P$ to $Q$
as $(P',Q', R')$ is a shunt, which is a contradiction.
Therefore, we may assume that $u\neq t(P)$ and similarly $u' \neq t(Q)$. Furthermore the out-neighbour $v$ of $u$ in $P'$ is in $V(Q)$ and the out-neighbour $v'$ of $u'$ is in $V(P)$.
Since $P'$ and $Q'$ are disjoint, $P'[s(P'), u]  = P[s(P), u]$  and $Q'[s(Q'), u'] = Q[s(Q), u']$, it follows that
$C=\{uv, u'v'\}$ is a crossing with respect to $P$ and $Q$, and thus a tight crossing.

Consider the dipath $R'$.
\begin{itemize}
\item Assume first that $s(R')\in P'[s(P'),u]$.
Let $S$ be the shortest subdipath of $R'\cup Q'[t(R'), t(Q')]$ such that $s(S)=s(R')$ and $t(S)\in V(P)\cup V(Q)$.
Vertex $t(S)$ cannot be in $Q[s(Q), u']$ for otherwise $S = R'$ and it would be a dipath of length at least $2$ between $P$ and $Q$.
Furthermore, if $t(S) \in V(Q)$, $\{s(R')t(S), u'v'\}$ is a loose crossing, since the distance between $u'$ and $t(S)$ in $Q$ is at least $2$ ($u$ is between $s(R')$ and $v$, and $v$ is between $u'$ and $t(S)$).
Therefore $t(S) \in V(P)$ and so $t(S)$ is on $P[v', t(P)]$.
But then $S$ is a forward path or an arc bypass in $P$, a contradiction.

\item Assume now that $s(R') \in P'[v, t(P')]$.

Set  $P^*=Q[s(Q), u']\cup u'v \cup P'[v, t(P')]$  and $Q^*=P[s(P), u]\cup uv' \cup Q'[v', t(Q')]$.
If $t(R')\in Q'[v', t(Q')]$, then $(P^*, Q^*, R')$ is an $(\{s_1,s_2\}, \{t_1,t_2\})$-shunt. But $P^*$ and $Q^*$ have more arcs in common with $P$ and $Q$ than $P'$ and $Q'$, which contradicts our choice of $(P', Q', R')$.
Therefore $t(R')\in Q'[s(Q'), u']$.

Let $S$ be the shortest subdipath of $P'[v,s(R')]\cup R'$ such that $t(S)=t(R')$ and $s(S)\in V(P)\cup V(Q)$.

Assume first that $s(S)\in V(Q)$. Then $S$ is a $C$-backward path. Hence it must have length $1$.
Therefore $s(S)\notin V(P')\cup V(Q')$ because $R'$ has length at least $2$.
Let $u_1$ be the farthest vertex on $P'[v, t(P')]$ that is in $V(Q)$ and such that $P'[v, u_1]$ does not intersect $P$.
Observe that $u_1$ appears before $s(S)$ in $Q$, for otherwise there would be a $C$-bypass in $P'$, as $s(S) \notin P'$.
In particular, $u_1$ is not the terminal vertex of $P'$.
Let $v_1$ be the first vertex after $u_1$ along $P'$ which is on $P\cup Q$.
It must be in $V(P)$ by the choice of $u_1$.
Therefore $u_1v_1$ is an arc because there is no dipath of length at least $2$ between $Q$ and $P$.
Let $u_2$ be the farthest vertex on $Q'[v', t(Q')] \cap P$ such that $Q'[v', u_2]$ does not intersect $Q$.
Then $v_1$ is after $u_2$ along $P$, for otherwise there would be an arc bypass in $P$ for $u_1v_1$. Thus $u_2$ is not the terminal vertex of $Q'$. Let $v_2$ be the first vertex after $u_2$ along $Q'$ which is on $P\cup Q$.
It must be in $V(Q)$ by the choice of $u_2$.
Hence $u_2v_2$ is an arc because there is no dipath of length at least $2$ between $P$ and $Q$.
Moreover, observe that for every vertex $x$ in $Q[v,u_1]-P'$ there is a subdipath of $P'$ which is an $x$-bypass.
Therefore $v_2$ must be in $Q]u_1, t(Q)]$ for otherwise it would be an arc bypass.
Hence $\{u_2v_2, u_1v_1\}$ is a crossing for $P \cup Q$, and so it must be tight.
This implies in particular that $s(S)\in Q[v_2, t(Q)]$.

Set $P^+=P'[s(P'),u]\cup uv' \cup Q'[v', u_2]\cup u_2v_1 \cup P'[v_1, t(P')])$ and
$Q^+=Q'[s(Q), u'] \cup u'v \cup P'[v, u_1] \cup u_1v_2 \cup Q'[v_2, t(Q')])$.
If $s(R')\in  P'[v_1, t(P')])$, then $(P^+, Q^+, R')$ is an  $(\{s_1,s_2\}, \{t_1,t_2\})$-shunt. But
$P^+$ and $Q^+$ have more arcs in common with $P$ and $Q$ than $P'$ and $Q'$, which contradicts our choice of $(P', Q', R')$.
Therefore $s(R')\in P'[(v,u_1)]$. Now $P'[v, s(R')]\cup R'$ contains a subdipath $T$ that  is internally disjoint from $P$ and $Q$  and has initial vertex in $Q[v,u_1]$ and terminal vertex in $P\cup Q[v_2, t(Q)]$, by the locations of $s(R')$ and $s(S)$, $s(S)$ being in $R'$, and since we know that $u_1 \in P'$, $v_2 \in Q'$ and therefore neither of them is in $R'$.
Necessarily, $t(T)\in V(P)$ for otherwise $T$ is an arc bypass. Hence $T$ is an arc.
Furthermore, $t(T)$ could not be in $P[v',u_2]$ for otherwise $Q'$ would contain a $t(T)$-bypass, which would be an arc bypass.
Hence $t(T)\in P]v_1, t(Q)]$ and $\{u_2v_2, T\}$ is a loose crossing, a contradiction.

\medskip

Assume now that $s(S) \in V(P)$.Then it must be in $P[v', t(P)]$.
Since there is no dipath of length at least $2$ from $P$ to $Q$, $S$ has length $1$.
Moreover, since $R'$ has length at least $2$, $s(S)$ is an internal vertex of $R'$, so it is not in $V(P'\cup Q')$.
Let $u_2$ be the farthest vertex on $Q'[v', t(Q')]$ that is in $V(P)$ and such that $Q'[v', u_2]$ does not intersect $Q$.
Then $u_2$ appears before $s(S)$ on $P$, for otherwise there would be an arc bypass for $s(S)t(S)$ in $P$ and so $u_2$ is not the terminal vertex of $Q'$.
Let $v_2$ be the first vertex after $u_2$ along $Q'$ which is on $P\cup Q$.
It must be in $V(Q)$ by the choice of $u_2$, and so on $Q[v, t(Q)]$.
$u_2v_2$ is an arc for otherwise for otherwise there would be a dipath of length $2$ from $P$ to $Q$.
Let $u_1$ be the farthest vertex on $P'[v, t(P')]$ that is also in $V(Q)$ such that $P'[v, u_1]$ does not intersect $P$.
Vertex $u_1$ appears before $v_2$ in $Q$, for otherwise there would be an arc bypass for $u_2v_2$ in $Q$, and so $u_1$ is not the terminal vertex of $P'$.
Let $v_1$ be the first vertex after $u_1$ along $P'$ which is on $P\cup Q$.
It must be in $V(Q)$ by the choice of $u_1$.
Hence $u_1v_1$ is an arc because there is  no dipath of length at least $2$ between $Q$ and $P$.
Moreover, observe that for every vertex $x$ in $P[v',u_2]-Q'$ there is a subdipath of $P'$ which is an $x$-bypass.
 Therefore $v_1$ must be in $P]u_2, t(P)]$ for otherwise it would be an arc bypass.
 Hence $\{u_2v_2, u_1v_1\}$ is crossing for $P \cup Q$, and so it must be tight.
This implies in particular that $s(S)\in P[v_1, t(P)]$.

We then find a contradiction as in the previous case by considering $P^+$ and $Q^+$.
\end{itemize}
This finishes the proof of the Lemma.\end{proof}

\begin{theorem}\label{lem:shunt}
{\sc shunt} can be solved in $O(n^2 (n+m))$ time.
\end{theorem}
\begin{proof}
We describe a procedure {\tt shunt}$(D, s_1, s_2, t_1, t_2)$, solving {\sc shunt} and estimate its time complexity.
The procedure then check, by a Menger algorithm, if there are two disjoint dipaths $P, Q$ from $\{s_1, s_2\}$ to $\{t_1, t_2\}$, which runs in $O(n+m)$ time.
Observe that the arcs $s_1s_2$ and $s_2s_1$ are useless, so we remove them from $D$ if they exist.
Then we should check if there are paths of length at least $2$, arc bypasses, loose crossings, $C$-forward paths, backward paths of length at least $2$ or crossing bypasses with respect to $P$ and $Q$, according to  Lemma~\ref{lem:shunt-equiv}.
For every vertex $u \in P$ (and any vertex in $Q$, similarly), we do the following: if $u$ has a neighbour in $Q$, we test if there is a path from $P[s(P),u[$ to $P]u, t(P)]$, which would be an arc bypass.
Let $v'$ be the last  vertex of $Q$ such that $uv'$ is an arc (and such that $v'u$ is an arc, similarly).
Then, for a vertex $v$ in $P]u,t(P)]$, we check if there is a vertex $u'$ in $Q[s(Q), v'[$ such that $u'v$ ($vu'$) is an arc.
Then if $u,v$ and $u'v'$ have distance at least $2$ in $P$ and $Q$ respectively, it would be a loose crossing.
Otherwise, if such edges exists there is a tight crossing $C = \{uv', u'v\}$ containing $u$.
We then run a Menger algorithm one more time, to test if there is a dipath from $u$ to $v$ in the digraph induced by $(V(D)-V(P)-V(Q))$, which would be a forward path.
So far, the running time of the algorithm is bounded by $O(n^2 (n+m))$: the complexity of calculating the $P$ and $Q$ initially plus the complexity of, for each vertex in in $P \cup Q$, look for an arc bypass, plus the running time of analysing if each pair of vertices in $P$ or $Q$ are part of a loose crossing and finally plus the time of looking for a forward path.
Then, still considering the same tight crossing $C$, for every vertex $x$ in $P[v, t(P)]$, we check if there is a dipath to some $y$ in $P[v, t(P)]$.
If it is the case and $xy$ is an arc, we then look for dipaths from $P[s(P),y[$ to $P]y,u]$ and from $P[v,x[$ to $P]x,t(P)]$.
This can be done in $O(n^2 (n+m))$: for every pair of vertices $u$ and $x$, we uses Menger algorithm possibly three times to compute the dipaths above. So, {\tt shunt}$(D, s_1, s_2, t_1, t_2)$ runs in $O(n^2 (n+m))$ time in total.
\end{proof}

With Theorem~\ref{lem:shunt} at hands, we now deduce Theorem~\ref{E_9}. We believe that it could also be used to prove the tractability of other digraphs $F$.

\begin{proof}[Proof of Theorem~\ref{E_9}]

For every vertex $v$ of $D$ and for every set of two outneighbours $s_1, s_2$ and two inneighbours $t_1, t_2$ of $v$, we check if  there is a $(\{s_1,s_2\}, \{t_1,t_2\})$-shunt in $D$. Observe that there is an $E_9$-\pb in $D$ in which $v$ is the $a$-vertex if and only if there is a shunt for a pair of outneighbours and a pair of  inneighbours of $v$. So, since there are $n^5$ possible choices for vertex $v$ and its neighbours, and for each of them we apply the procedure {\tt shunt} that runs in $O(n^2 (n+m))$ time, our algorithm decides whether there is an $E_9$-\pb\ in $D$ in $O(n^7 (n+m))$ time.
\end{proof}

\section{Proof of Theorem~\ref{thm:main}}\label{sec:classify}

To prove Theorem~\ref{thm:main}, we review all digraphs $D$ of order $4$, and
determine if they are tractable or intractable or if their status is unknown.

For a digraph $D$, its {\it $2$-cycle graph} $G_D$ is the graph with the same vertex set in which two vertices are linked by an edge if they are in a directed $2$-cycle in $D$.
Thus, the $2$-cycle graph of an oriented graph is an empty graph.
We denote by $A'(D)$ be the set of arcs of $D$ which are not in directed $2$-cycles.

Let $F$ be a digraph of order $4$.
By Corollary~\ref{NP-2cycle}, if $F$ contains a directed $2$-cycle whose vertices are big, then $F$ is intractable. So we may assume that $F$ contains no such $2$-cycles.
In particular, it implies that $G_F$ has at most one vertex of degree at least two.
So $G_F$ has at most three edges.

\medskip

\noindent\underline{Case 0}: $G_F$ has no edges.
Then $F$ is tractable by Theorem~\ref{orient-4}.

\medskip

\noindent\underline{Case 1}: $G_F$ has three edges. Then necessarily, $G_F$ is the star of order $3$. Hence $F$ is either the symmetric star or the superstar of order $4$. In both cases, $F$ is tractable, see Subsection~\ref{subsec:easier}.

\medskip

\noindent\underline{Case 2}: $G_F$ has exactly two edges which are non-adjacent.

If $|A'(F)|\leq 1$, then $F$ has no big vertex, so by Corollary~\ref{cor:nobig} $F$ is tractable.

If $|A'(F)|\geq 2$, then $F$ is either one of $N_1$, $N_2$, $N_3$, $N_4$, $O_1$ and their converses, or $F$ has no big vertex.
In the later case, $F$ is tractable by Corollary~\ref{cor:nobig}.
If $F=N_i$ for $i\in \{1,2,3,4\}$, then $F$ is intractable by Proposition~\ref{propNP}.
We do not know the complexity of $O_1$-subdivision.

\medskip

\noindent\underline{Case 3}: $G_F$ has exactly two edges which are adjacent.

If $A'(F)$ is empty, then $F=SS_2+K_1$, where $K_1$ is the digraph on one vertex.  As discussed in Subection~\ref{sec:easy}, $SS_2$ is tractable. Thus, by Lemma~\ref{lem:union-spider}, $F$ is tractable.

If $|A'(F)|=1$, then $F$  either is $SS^*_2+K_1$, or $E_1$ or the converse of $E_1$, or is obtained from $SS_2^*$ by gluing an arc on its centre. Now $SS^*_2+K_1$ is tractable by Corollary~\ref{cor:superstar} and Lemma~\ref{lem:union-spider}; $E_1$ (and thus its converse) is tractable by Proposition~\ref{poly};
if $F$ is obtained from $SS_2^*$ by gluing an arc on its centre, then it is tractable by Theorem~\ref{superstar} and by Lemma~\ref{lem:add-spider}.

If $|A'(F)|=2$, then $F$ is either $E_2$, $E_3$, $E_9$, $O_2$ or one of their converses.
If $F\in \{E_2, E_3, E_9\}$, then it is tractable by Proposition~\ref{poly}. If $F= O_2$, then we do not know.

If $|A'(F)|=3$, then $F$ is either $N_5$, $N_6$, $O_3$ or one of their converses.
If $F \in \{N_5,N_6\}$, then it is intractable by Proposition~\ref{propNP}. The complexity of $O_3$-\pb\ is still unknown.

\medskip

\noindent\underline{Case 4}: $G_F$ has exactly one edge.

If $F$ has no big vertices, then, by Corollary~\ref{cor:nobig}, $F$ is tractable.
Henceforth, we may assume that $F$ has a big vertex, i.e. a vertex with in-degree or out-degree at least $3$ or both in-degree and out-degree equal to $2$.
Observe that it implies that $F$ is connected and $|A'(F)|\geq 2$.

$|A'(F)|=2$, $F$ obtained from $\vec{C_2}$ by gluing a spider on one its vertices. Then $F$ is tractable by Lemma~\ref{lem:add-spider}.

If $|A'(F)|=3$, then we distinguish several subcases according to the position of the arcs of $A'(F)$ relatively to the directed $2$-cycle $C$ of $F$.

\begin{itemize}
\item $A'(F)$ induces an orientation of a star.
Then $F$ is obtained from $W_2$ or its converse by gluing an arc on its centre.
Thus  $F$ is tractable by Lemma~\ref{lem:add-spider}.

\item $A'(F)$ induces an oriented path whose first vertex is a vertex of $C$ and whose third vertex is the other vertex of $C$.
Then $F$ is obtained either from the bispindle $B(2,1;1)$ by gluing an arc on one of its nodes, or from $W_2$ or its converse by gluing an arc on one of its vertices.
In both cases, $F$ is tractable by Lemma~\ref{lem:add-spider} and Theorem~\ref{bispindle} and Lemma~\ref{W2}.

\item $A'(F)$ induces an oriented $3$-cycle.
If this cycle is directed, then $F$ is a {\it windmill}, that is a subdivision of a symmetric star.
Bang-Jensen et al. \cite{BHM} proved that windmills are tractable, so $F$ is tractable.
If this cycle is not directed, then $F$ is either $E_4$ or its converse, or $N_{7}$.
If $F$ is $E_4$ or its converse, then it is tractable by Proposition~\ref{poly}.
If $F=N_{7}$, then it is intractable by Proposition~\ref{propNP}.
\end{itemize}

If $|A'(F)|=4$, then it is either $N_8$, $N_{9}$, $E_5$, $E_6$, $E_7$, $E_8$, $O_4$,  $O_{5}$,  or one of their converses.
If $F$ is  $N_8$ or $N_{9}$, then it is intractable by Proposition~\ref{propNP}.
If $F$ is $E_5$, $E_6$, $E_7$, or $E_8$,  then it is tractable by Proposition~\ref{poly}.
The complexity of $O_4$-\pb\ and $O_5$-\pb\ is still open.

This concludes the proof of Theorem~\ref{thm:main}.

\end{document}